\documentclass[leqno, 11pt, a4paper]{amsart}
\usepackage[backend=biber,url=false, isbn=false, doi=false]{biblatex}
\AtEveryBibitem{\clearfield{month}}
\AtEveryBibitem{\clearfield{day}}
\usepackage{amsthm, amsmath, dsfont, mathrsfs, graphicx, amssymb,eucal, hyperref}

\newtheorem{thm}{Theorem}

\newtheorem{prop}{Proposition}
\newtheorem{lem}{Lemma}
\newtheorem{rem}{Remark}
\theoremstyle{definition}
\newtheorem{defi}{Definition}
\newtheorem{ex}{Example}

\newcommand{\cR}{\mathcal{R}}

\newcommand{\cC}{\mathcal{C}}
\newcommand{\cZ}{\mathbf{Z}}
\newcommand{\R}{\mathbb{R}}
\newcommand{\C}{\mathscr{C}}
\newcommand{\Q}{\mathbb{Q}}
\newcommand{\E}{\mathbb{E}}
\newcommand{\N}{\mathbb{N}}
\newcommand{\Pro}{\mathbb{P}}
\newcommand{\cF}{\mathcal{F}}
\newcommand{\eps}{\varepsilon}
\newcommand{\cP}{\mathcal{P}}
\newcommand{\cM}{\mathcal{M}}
\newcommand{\supp}{\textsf{supp}}

\newcommand{\ox}{\overline{x}}
\newcommand{\oy}{\overline{y}}
\newcommand{\law}{\textsf{Law}}
\newcommand{\pr}{{\rm pr}}

\renewcommand{\epsilon}{\varepsilon}
\newcommand{\wto}{\rightharpoonup}
\newcommand{\cto}[1]{\xrightarrow[#1]{}}
\newcommand{\scal}[2]{\langle #1, #2 \rangle}
\newcommand{\dd}{\mathrm{d}}
\newcommand{\extendedR}{\R \cup \{+\infty\}}

\title[A land of monotone plenty, \textit{bis repetita}]{A land of monotone plenty, \textit{bis repetita}: from classical to weak optimal transport}
\author{Virginie Ehrlacher}
\address{CERMICS, Ecole nationale des ponts et chaussées, IPP, CNRS \& INRIA}
\email{virginie.ehrlacher@enpc.fr}
\author{Rodrigue Lelotte}
\address{SAMM, Université Paris 1 Panthéon-Sorbonne, Paris, France.}
\email{rodrigue.lelotte@univ-paris1.fr}
\author{Luca Nenna}
\address{Universit\'e Paris-Saclay, CNRS, Laboratoire de math\'ematiques d'Orsay, ParMA, Inria Saclay, 91405, Orsay, France \& Institut Universitaire de France (I.U.F.)}
\email{luca.nenna@universite-paris-saclay.fr}

\begin{document}
	
\maketitle

\begin{abstract}
The celebrated $c$-cyclical monotonicity property is shown to boil down to the zeroth-order optimality condition for the optimal transport problem. More precisely, we show that optimality is equivalent to the non-negativity of the linear transport cost functional on the radial cone of admissible perturbations. We then utilise this point of view to extend the $c$-cyclical monotonicity property to the weak optimal transport problem, for which it corresponds to the first-order optimality condition, namely to the non-negativity of the linearisation of the weak transport cost functional near the optimiser. Altogether, this sheds new light on this monotonicity concept. For both  classical and weak optimal transport, we show that this property characterises (under suitable assumptions) optimal transport plans. In the classical case, we recover known results of the literature but with revisited proofs.   
\end{abstract}

\setcounter{tocdepth}{2}
\tableofcontents

\section{Introduction}

The notion of \textit{$c$-cyclical monotonicity} occupies a central role in optimal transport theory. A transport plan is said to be $c$-cyclically monotone if, in essence, local re-routing of the mass on the support of the transport plan leads to an increase of the overall cost of transport. Under suitable assumptions, this property is both necessary and sufficient for optimality. We refer the reader to the recent survey of De Pascale, Kausamo, and Wyczesany \cite{depascale60YearsCyclic2025} for a broad overview of the subject.

\subsection{Contributions} 
The purpose of the present work is twofold. First, we revisit cyclical monotonicity from a perturbative perspective based on elementary optimality considerations. In particular, we show that this property characterises under suitable assumptions optimal transport plans. Although we recover known results of the literature, the proofs are revisited in light of our approach. Second, building on this approach, we extend the notion of $c$-cyclical monotonicity to the case of the \textit{weak} optimal transport problem by introducing a notion of monotonicity which, contrary to previously existing notions in the literature, genuinely recovers the classical $c$-cyclical monotonicity property as a special case. 

\subsection{\textit{The gist of it}}
Given two Polish spaces $X$ and $Y$, two probability measures $\mu \in \cP(X)$ and $\nu \in \cP(Y)$, and a Borel cost function $c:X\times Y\to \extendedR$, the classical optimal transport problem reads
$$
\inf_{\pi \in \Pi(\mu, \nu)}  \C(\pi), \qquad \C(\pi) :=  \int_{X \times Y} c(x, y)\, \pi(\dd x, \dd y),
$$
where $\Pi(\mu,\nu)$ denotes the set of transport plans with marginals $\mu$ and $\nu$. Our starting point is deceptively simple. If $\pi$ is an optimal transport plan and if $\eta$ is an admissible perturbation preserving the marginal constraints, that is a signed Radon measure $\eta \in \cM(X \times Y)$ so that $\pi + t\eta \in \Pi(\mu, \nu)$ for $t>0$ small enough, then it obviously holds
$$
\C(\pi+t\eta)\geq \C(\pi).
$$
This is just the \textbf{zeroth-order optimality condition}. In the classical optimal transport problem, the cost functional is linear, and therefore optimality is equivalent to the non-negativity condition
$$
\C(\eta)\geq 0
$$
for every admissible perturbation $\eta$. Now, it will follow from elementary manipulations that the $c$-cyclical monotonicity exactly rephrases as demanding that $\C$ be non-negative for a special kind of ``cyclical" perturbations $\eta$. \textit{The gist of our work is to show that $\C$ being non-negative on the set made of these special perturbations is sufficient for it to be non-negative over the entire set of admissible perturbations and \textit{vice-versa}}, thus yielding sufficiency and necessity of $c$-cyclical monotonicity for optimality in the optimal transport problem. This gives a new and conceptually transparent route to the equivalence between optimality and cyclical monotonicity. Moreover, the same approach can be used to derive the analogue of $c$-cyclical monotonicity in other types of optimal transport problems.

Indeed, a central contribution of the paper concerns the weak optimal transport problem. We recall that this problem reads
$$
\inf_{\pi\in\Pi(\mu,\nu)}
\int_X C(x,\pi_x)\,\mu(\dd x),
$$
where $(\pi_x)_{x\in X}$ denotes the disintegration of $\pi$ with respect to its first marginal and the function $C:X\times \cP(Y)\to \extendedR$ is (typically) convex in its second argument.

A notion coined as \textit{$C$-monotonicity} was introduced in the literature by Backhoff-Veraguas, Beiglb\"ock, and Pammer \cite{backhoff-veraguasExistenceDualityCyclical2019} and shown to be necessary and sufficient for optimality under suitable assumptions. However, this notion does not fully recover classical cyclical monotonicity when the weak transport problem reduces to the ordinary optimal transport problem (see Example~\ref{ex:c-mon-limitation} \textit{infra}). One of the contributions of the present work is to introduce a notion of monotonicity for weak optimal transport, which we call \textit{$C$-cyclical monotonicity}, that genuinely extends the classical notion of $c$-cyclical monotonicity to the weak setting. This notion is formulated in terms of the gradients (or subgradients) of the weak cost with respect to its measure argument. Concretely, if $\rho\mapsto C(x,\rho)$ is differentiable in an appropriate sense (see Section~\ref{sec:diff}), then our monotonicity condition takes the form
$$
\boxed{\sum_{i=1}^n \nabla_\rho C(x_i,\pi_{x_i})(y_i) \leq \sum_{i=1}^n \nabla_\rho C(x_i,\pi_{x_i})(y_{i+1}).}
$$
for all $n \in \N$ and all families $(x_1, y_1), \dots, (x_n, y_n) \in \Gamma$ where $\Gamma \subset X \times Y$ is a measurable set so that $\pi(\Gamma) =1$. Here, we let $y_{n+1} := y_1$. This property enjoys several important features.  First, it exactly recovers the classical $c$-cyclical monotonicity when the weak problem boils down to the classical one, namely when the cost of transport is given by
$$
C(x,\rho) :=\int_Y c(x,y)\,\rho(\dd y).
$$
Indeed, in this case the gradient is given by $\nabla_\rho C(x,\rho)(y)=c(x,y)$, and one recognises that the above condition is exactly the $c$-cyclical monotonicity property. Second, we show that $C$-cyclical monotonicity is equivalent to optimality under assumptions comparable to those used in the classical setting. 

Our $C$-cyclical monotonicity property is derived by looking at the \textbf{first-order optimality condition} for the weak optimal transport. Indeed, the objective of this problem is
$$
\C(\pi) := \int_{X} C(x, \pi_x) \mu(\dd x),
$$
which, contrary to the classical optimal transport problem, is no longer linear, so we need to go to the first-order to say anything meaningful. Whereas in the classical problem, optimality is encoded by the non-negativity of the cost on admissible perturbations, in the weak setting, the same mechanism survives after replacing the transport cost by the linearised quantity
$$
\eta\mapsto \int_{X\times Y} \nabla_\rho C(x,\pi_x)(y)\,\eta(\dd x,\dd y).
$$
The $C$-cyclical monotonicity introduced above is then derived from the non-negativity of the above functional once evaluated at these special ``cyclical'' perturbations already mentioned above. 

\subsection{Previous works} 
It is possible that part of our arguments \textit{may} not be completely new to some readers already familiar with the subject. In a way, a part of what we present in this work can be seen as a way of explaining what cyclical monotonicity is all about in a different perspective with a different terminology, thus shedding a new kind of light on this concept.  Whether or not these ideas are known to specialists, they are certainly not widely familiar to a broader audience. During our investigation, we came across a couple of works of S. Bianchini and L. Caravenna  \cite{bianchiniOptimalityCcyclicallyMonotone2010, bianchiniExtremalityUniquenessOptimalityb} which share some common features with our own work here and that should be readily advertised. We shall also mention the work of \textcite{beiglbockCyclicalMonotonicityErgodic2015} which utilises the ergodic theorem to prove sufficiency of the $c$-cyclical monotone property for optimality somehow in the same way as we do. Also -- as the title of our paper is obviously a nod to this work — we shall also mention the seminal paper \cite{beiglbockLandMonotonePlenty2019b} — and also \cite{zaevMongeKantorovichProblem2015}. However, the unified formalism presented here and its use as a systematic recipe to derive the monotonicity property to other type of optimal transports is  decidedly new --- at least to the best of our knowledge.

\subsection{Outline of the paper}
Section \ref{sec:OT} revisits the $c$-cyclical monotonicity property in the classical optimal transport through the perturbative point of view. We introduce the radial cone of admissible perturbations and show that optimality is equivalent to the zeroth-order optimality condition on this cone. We then relate this condition to finite optimality and to $c$-cyclical monotonicity. Section \ref{sec:WOT} is devoted to weak optimal transport. We derive the first-order optimality conditions for weak transport problems, and use them to define the notion of $C$-cyclical monotonicity and establish its equivalence with optimality under suitable assumptions. We also discuss the case of the barycentric cost (which is non-differentiable). 

\section{A perturbative approach to the classical optimal transport}\label{sec:OT}

\subsection{Setting and notations} 
For a Polish space $Z$, we write $\cM(Z)$ the set of finite (and signed) Radon measures on $Z$. By a \textit{finite} Radon measure $\eta$, it is meant that $\Vert \eta \Vert_{TV} < \infty$ where the total variation norm $\Vert \eta \Vert_{TV}$ is defined as $\Vert \eta \Vert_{TV} := |\eta|(Z)$ where $|\eta| := \eta_+ + \eta_-$ with $\eta_+$ and $\eta_-$ being respectively the positive and negative parts of the signed measure $\mu$. The space $\cM(Z)$ is equipped with the \textit{vague topology}, that is in duality with $C_0(Z)$, \textit{i.e.} the set of real-valued continuous functions on $Z$ vanishing at infinity --- when $Z$ is compact, this set is simply the set of all continuous functions. We recall that finite radii balls $\{ \eta \in \cM(Z): |\eta|(Z) \leq r\}$ for $0 \leq r < \infty$ are metrisable in this topology provided that $Z$ is Polish and furthermore locally compact, see \textit{e.g.} \cite[Theorem 3.29]{brezisFunctionalAnalysisSobolev2011}. Finally, we write $\cP(Z) \subset \cM(Z)$ be the subset of probability measures.

\medskip

Let $\mu \in \cP(X)$ and $\nu \in \cP(Y)$ be probability measures  on Polish (and locally compact) spaces $X$ and $Y$, and let  $c : X  \times Y \to \extendedR$ be a Borel-measurable cost of transport possibly assuming $+\infty$ as value. We consider the classical optimal transport problem 
\begin{equation}\label{OT}\tag{OT}
	\min_{\pi \in \Pi(\mu, \nu)}  \C(\pi), \qquad \C(\pi) :=  \int_{X \times Y} c(x, y)\, \pi(\dd x, \dd y)
\end{equation}
where $\Pi(\mu, \nu)$ is the standard set of \textit{transport plans}, that is those probability measures $\pi \in \cP(X \times Y)$ so that the first marginal of $\pi$ (\textit{i.e.} with respect to $X$) is $\mu$ and such that the second marginal of $\pi$ (\textit{i.e.} with respect to $Y$)  is $\nu$. This also rewrites $\pr_1^\sharp \pi = \mu$ and $\pr_2^\sharp \pi = \nu$ where $\pr_1 : X \times Y \to X$ and $\pr_2 : X \times Y \to Y$ are the canonical surjections of the Cartesian product $X \times Y$ onto its factors $X$ and $Y$ — and where we use the standard notation $f^\sharp \pi$ for any measurable map $f : X \times Y \to Z$  to mean the push-foward measure defined as $f^\sharp \pi(A) := \pi(f^{-1}(A))$ for all Borel sets $B \subset Z$.

\subsection{The zeroth-order optimality condition for \eqref{OT}}
We start with the following simple observation. If $\pi \in \Pi(\mu, \nu)$ is optimal for the optimal transport problem \eqref{OT} and $\pi$ has finite transport cost, meaning that $\C(\pi) < \infty$, then it must be that $\C(\pi + t \eta) \geq \C(\pi)$ for all $t \in \R$ and $\eta \in \cM(X \times Y)$ such that $\pi + t \eta$ is still an admissible transport plan, that is $\pi + t \eta \in \Pi(\mu, \nu)$. By linearity of the cost functional $\C$, this is equivalent to saying that $\C(\eta) \geq 0$ for such perturbations $\eta$ provided that $t > 0$. Let us then introduce for $\pi \in \Pi(\mu, \nu)$ the set $\cR_\pi \subset \cM(X \times Y)$ defined as
\begin{equation}\label{radial_cone}
	\boxed{\cR_\pi := \bigg\{\eta \in \cM(X \times Y) : \pi + t \eta \in \Pi(\mu, \nu) \text{ for some } t > 0 \bigg\}.}
\end{equation} 
Then, the previous statement rephrases as follows: \textit{If $\pi$ is optimal for the optimal transport problem \eqref{OT} and has finite transport cost, that is $\C(\pi) < \infty$, then $\C(\eta) \geq 0$ for all $\eta \in \cR_\pi$}. In the context of convex analysis, the set $\cR_\pi$ is referred to as the \textit{radial cone} to the (convex) set $\Pi(\mu, \nu)$ at the point $\pi$. We refer to Remark~\ref{rem:zeroth-order-opt-cond} and \eqref{eq:radial-cone} \textit{infra} for a general definition of the radial cone to a convex set. Furthermore, the converse statement holds as well: \textit{If $\pi \in \Pi(\mu, \nu)$ is such that $\C$ is non-negative on $\cR_\pi$, then $\pi$ is optimal in \eqref{OT}} — again, provided that $\C(\pi)< \infty$. Indeed, for all $\gamma \in \Pi(\mu, \nu)$, one has $\C(\gamma) = \C(\gamma - \pi + \pi) = \C(\gamma - \pi) + \C(\pi) \geq \C(\pi)$ since $\eta := \gamma-\pi$ is  obviously an element of $\cR_\pi$. Thus $\pi$ is optimal. Altogether, we obtain the following proposition which is just the \textbf{zeroth-order optimality condition} for \eqref{OT}:
\begin{prop}\label{prop:zero-order-cond}
	A transport plan $\pi \in \Pi(\mu, \nu)$ with finite cost $\C(\pi) < \infty$ is optimal in \eqref{OT} if and only if $\C(\eta) \geq 0$ for all $\eta \in \cR_\pi$, where $\cR_\pi$ is the radial cone to the feasible set $\Pi(\mu, \nu)$ at the point $\pi$ as defined in \eqref{radial_cone}.
\end{prop}

\begin{rem}\label{rem:zeroth-order-opt-cond}
	Proposition~\ref{prop:zero-order-cond} is just one particular instance of a much general situation. Let $X$ be a normed vector space, $L : X \to \extendedR$ be a linear function on $X$ and $C \subset X$ be a convex set. Let us consider the minimisation problem
	\begin{equation}\label{LP}
		\min_{x \in C} Lx.
	\end{equation}
	Then, $x \in C$ is optimal for \eqref{LP} if and only if  (provided that $Lx < \infty$) $Lh \geq 0$ for all $h \in R_C(x)$, where the \textit{radial cone} $R_C(x) \subset X$ (``\textit{to the convex set $C$ at the point $x$'}') is defined as
	\begin{equation}\label{eq:radial-cone}
		\cR_C(x) := \left\{h \in X : x + th \in C \text{ for some $t > 0$}\right\}.
	\end{equation}
	(We remark passing by that by convexity of $C$, if there exists $t_0 > 0$ such that $x + t_0 h \in C$ then $x + th \in C$ for all $0 \leq t \leq t_0$.) Now, let us assume that $x$ is optimal, and let $h \in \cR_C(x)$. Then, there exists a small enough $t > 0$ such that $x + th \in C$. By optimality of $x$ (and the fact that $Lx < \infty$), we have $L(x + th) \geq L(x)$ and thus $L(h) \geq 0$ by linearity of $L$. The other direction is also straightforward. Indeed, assume that $L(h) \geq 0$ for all $h \in \cR_C(x)$. Then for all $y \in C$, we have $L(y) = L(x) + L(y -x) \geq L(x)$ since $y-x \in \cR_C(x)$, and thus $x$ is optimal for the problem \eqref{LP}.
\end{rem}

Let us now characterise the radial cone $\cR_\pi$ defined above in \eqref{radial_cone} explicitly. We have:
\begin{lem}[Characterisation of the radial cone $\cR_\pi$]\label{lem:car_r_pi}
	Let any $\pi \in \Pi(\mu, \nu)$. The radial cone $\cR_\pi$ to the set of transport plan $\Pi(\mu, \nu)$ at the point $\pi$ — as defined in \eqref{radial_cone} — is characterised as the set of \emph{signed} Radon measure $\eta = \eta^+ - \eta^-$ such that $\eta^+$ and $\eta^-$ share common marginals, that is $\pr_i^\sharp \eta^+ = \pr_i^\sharp \eta^-$ for $i = 1, 2$, and such that $\eta^- \ll \pi$ and $\frac{d \eta^-}{d \pi} \in L^\infty(\pi)$. 
\end{lem}
Here, we write $\eta^-$  the negative part of $\eta$ and correspondingly $\eta^+$ its non-negative part. We use the standard notation $\eta^- \ll \pi$ to mean that $\eta^-$ is absolutely continuous with respect to $\pi$, and we write $\frac{\dd \eta^-}{\dd \pi}$ the Radon-Nykodym density of $\eta^-$ with respect to $\pi$. We remark that, in Lemma~\ref{lem:car_r_pi}, the fact that $\eta^+$ and $\eta^-$ share common marginals also rephrases as $\pr_i^\sharp \eta =0$ for $i = 1, 2$, that is the (signed) measure $\eta$ has \textit{zero} marginals. 

\begin{proof}[Proof of Lemma~\ref{lem:car_r_pi}]
	Let $\cR'_\pi \subset \cM(X \times Y)$ be the set 
		\begin{equation}\label{eq:R_pi}
				\cR'_\pi = \left\{\eta \in \cM(X \times Y) \,\middle|\, \pr_i^\sharp \eta = 0 \text{ for }i = 1,2, \quad \frac{\dd \eta^-}{\dd \pi} \in L^\infty(\pi) \right\}.
		\end{equation}
		— where it is implicitly understood that $\eta^- \ll \pi$. Let us show that $\cR_\pi = \cR'_\pi$, which is exactly the thesis of Lemma~\ref{lem:car_r_pi}. Let $\eta \in \cR'_\pi$ and let us define $\pi_t := \pi + t \eta$. We need to show that for a small enough $t > 0$, the measure $\pi_t$ is in the set of transport plans $\Pi(\mu, \nu)$. This means showing that $\pi_t$ is positive and that it verifies the marginal constraints -- note that, if we show these two properties hold, it will immediately follow that this measure is in fact a \textit{probability} measure. Let us decompose $\eta := \eta^+ - \eta^-$. By assumption, the negative part $\eta^-$ is absolutely continuous with respect to $\pi$ and its Radon-Nykodym density is in $L^\infty(\pi)$. Therefore, $\pi_t$ is positive for all $0 \leq t \leq t_0$ where we define $t_0^{-1} := \Vert \frac{d \eta^-}{d \pi} \Vert_{L^\infty(\pi)} > 0$. Now, the fact that $\pi_t$ verifies the marginal constraints is immediate from the fact that $\pr_i^\sharp \eta = 0$ for all $i = 1, 2$. This shows that $\cR'_\pi \subset \cR_\pi$. We now prove the converse inclusion. Let $\eta \in \cR_\pi$. Then, there exists a (small enough) $t > 0$ such that $\pi_t \in \Pi$, where we let $\pi_t$ as above. This (again) immediately implies that $\pr_i^\sharp \eta = 0$ for all $i = 1, 2$. Furthermore, because $\pi_t$ must be a positive measure, it must be that $t \eta^-(A) \leq \pi(A)$ for all Borel sets $A \subset X \times Y$, which means that $\eta^-$ is absolutely continuous with respect to $\pi$ and that its Radon-Nykodym density is uniformly bounded with respect to $\pi$ — and moreover $\Vert \frac{d\eta^-}{d \pi} \Vert_{L^\infty(\pi)} \leq t^{-1}$. Hence $\cR_\pi = \cR_\pi'$ and the thesis of Lemma~\ref{lem:car_r_pi} follows.
	\end{proof}

\subsection{On $c$-cyclical monotonicity and its related notions}
In this section, let us recall what $c$-cyclical monotonicity is all about and how it can be related to the perturbative approach outlined above. First, let us recall the mere definition of $c$-cyclical monotonicity:
\begin{defi}\label{def:c-cM}
	A set $\Gamma \subset X \times Y$ is said to be \textit{$c$-cyclically monotone} if for all couples of points $(x_1, y_1), \dots, (x_n, y_n)$ in $\Gamma$ — where $n \in \N$ is arbitrary (but finite) — we have
	\begin{equation}\label{eq:c-cM}
		\sum_{i = 1}^n c(x_i, y_i) \leq \sum_{i = 1}^n c(x_i, y_{i+1}),
	\end{equation}
	where we let $y_{n+1} := y_1$. A transport plan $\pi \in \Pi(\mu, \nu)$ is said to be $c$-cyclically monotone if it is \textit{concentrated }on a $c-$cyclically monotone set, that is if there exists $\Gamma \subset X \times Y$ which is $c$-cyclically monotone such that $\pi(\Gamma) = 1$.
\end{defi}
It is known that a property equivalent to that of $c-$cyclical monotonicity is the following:
\begin{defi}\label{def:fin-opt}
	A transport plan $\pi \in \Pi(\mu, \nu)$ is said to be \textit{finitely optimal} if for all couples of positive measures $\alpha$ and $ \alpha' \in \cM(X \times Y)$ with \textit{finite} supports sharing common marginals — that is $\pr_i^\sharp \alpha = \pr_i^\sharp \alpha'$ for $i = 1, 2$ — and such that the support of $\alpha$ is contained in that of $\pi$, it holds that $\C(\alpha) \leq \C(\alpha')$.
\end{defi} 
The fact that this property\footnote{This notion goes back to at least the work of \textcite{beiglbockLandMonotonePlenty2019b} — and \cite{beiglbockProblemOptimalTransport2016} in the case of the martingale optimal transport setting. Nevertheless, the journal version (\& first preprinted version) of \cite{beiglbockLandMonotonePlenty2019b} refers the reader to \cite[Ex. 2.21]{villani_optimal_2009} — whereas the second preprinted version of same article also refers to \cite{GradientFlows2008a} when introducing this notion. Bottom line, finite optimality has probably been around as \textit{folkloric knowledge} for quite a while.} implies $c$-cyclical monotonicity is rather straightforward. Indeed, let us assume that $\pi \in \Pi(\mu, \nu)$ is finitely optimal, and let $(x_1, y_1), \dots, (x_n, y_n) \in \Gamma$ for $n \in \N$ where $\Gamma := \supp(\pi)$. Then, let us define the positive and finitely supported measures $\alpha := \sum_{i = 1}^n \delta_{(x_i, y_i)}$ and $\alpha' := \sum_{i = 1}^n \delta_{(x_i, y_{i+1})}$ where we let $y_{n+1} := y_1$. By construction, we have $\supp(\alpha) \subset \Gamma$. It is also evident that $\alpha$ and $\alpha'$ share the same marginals — since $y_{n+1} = y_1$. Therefore, finite optimality says that $\C(\alpha) \leq \C(\alpha')$. But $\C(\alpha) = \sum_{i=1}^n c(x_i, y_i)$ and $\C(\alpha') = \sum_{i=1}^n c(x_i, y_{i+1})$ so that we exactly recover the $c$-cyclical monotonicity property.  The proof of the converse statement, namely that $c$-cyclical monotonicity is equivalent to the finite optimality property is to be found in \cite[Prop. 1.6]{pascaleSufficiencyCyclicalMonotonicity2025} or yet \cite{depascale60YearsCyclic2025} --- see also Remark~\ref{rem:equiv-fin-opt-c-cm}.

	\textit{Finite optimality} of Definition~\ref{def:fin-opt} and the \textit{zeroth-order optimality condition} shown in Proposition~\ref{prop:zero-order-cond} are obviously resembling. In the case of finite optimality, if we let $\eta := \alpha' - \alpha$ where $\alpha$ and $\alpha'$ are as in Definition~\ref{def:fin-opt}, then we see that $\eta$ is a specific kind of perturbation of the constraint set $\Pi(\mu, \nu)$ at the point $\pi$. Indeed, since $\alpha$ and $\alpha'$ share common marginals, the measure $\pi + \eta$ is a ``coupling'' of $\mu$ and $\nu$, in the sense that its marginals are precisely $\mu$ and $\nu$. Nevertheless, we stress out that it need \textit{not} be a rightful element of the radial cone $\cR_\pi$, because $\alpha$ is purely atomic and therefore has no reason \textit{a priori} to be absolutely continuous with respect to $\pi$. Another way to express this is by the fact that $\pi + \eta$ may not be a \textit{positive} measure in this situation.

	On the other side, in terms of the vocabulary of \cite{beiglbockLandMonotonePlenty2019b}, this means that the elements of the radial cone $\eta \in \cR_\pi$ are precisely those (signed) measure such that $\eta = \eta^+ - \eta^-$ and $\eta^+$ is a \textbf{competitor} for $\eta^-$. In essence, $\eta^+$ represents a re-routing of some amount of mass taken from the transport plan $\pi$, which is understood from the fact that $\eta^-$ is absolutely continuous with respect to $\pi$, that is $\eta^- \ll \pi$. Again, the difference here with the notion of competitors of \cite{beiglbockLandMonotonePlenty2019b} is that it is demanded here that $\eta^-$ be absolutely continuous, whereas in the case of finite optimality, it is demanded that $\eta^-$ be finitely supported inside the support of $\pi$.

\medskip

\begin{defi}[Finite perturbations]
    Let us define $\cF_\pi \subset \cM(X \times Y)$ the set of perturbations pertaining to finite optimality, namely the set of (signed) purely atomic measures $\eta$ with \textit{finite} support — meaning that $\supp(|\eta|)$ is a finite set where we recall that $|\eta| = \eta^+ + \eta^-$ — such that $\pr_i^\sharp \eta = 0$ for $i = 1, 2$ and such that $\supp(\eta^-) \subset \supp(\pi)$. In the language of \cite{beiglbockLandMonotonePlenty2019b}, this means that $\eta^+$ is a competitor to $\eta^-$. We will refer to elements of $\cF_\pi$ as \textit{finite perturbations}.
\end{defi}
By construction, it is obvious that \textbf{\textit{finite optimality as in Definition~\ref{def:fin-opt} is equivalent to demanding that $\C(\eta) \geq 0$ for all $\eta \in \cF_\pi$}}. We can also define the subset $\cC_\pi$ of $\cF_\pi$ of \textit{cyclical perturbations}, to wit the set of (signed) purely atomic measures $\eta$ of the form
	$$
	\eta = \frac1n \sum_{i = 1}^n (\delta_{(x_i, y_{i+1})} - \delta_{(x_i, y_i)}).
	$$
for some $n \in \N$ and family of points $(x_1, y_1), \cdots, (x_n, y_n) \in \supp(\pi)$ with $y_{n+1} := y_1$. It is then obvious that the \textbf{\textit{$c$-cyclical monotonicity property is exactly equivalent to the fact that $\C(\eta) \geq 0$ for all $\eta \in \cC_\pi$.} }

\begin{rem}[Equivalence between finite optimality and $c$-cyclical monotonicity]\label{rem:equiv-fin-opt-c-cm}
    We explained above that \emph{finite optimality} implies $c$-cyclical monotonicity. Within our language, this trivially follows from the fact that $\cC_\pi$ is a subset of $\cF_\pi$. The reverse statement also holds, and we refer the reader to \cite[Appendix D]{depascale60YearsCyclic2025}. We stress out that this necessitates no assumptions whatsoever on the cost of transport $c$. In fact, the rationale found in \cite{depascale60YearsCyclic2025} relies on purely algebraic manipulations, and in fact it can be stated abstractly as follows: \emph{If $L : \cM(X \times Y) \to \extendedR$ is a (not necessarily continuous) linear functional, and if $L(\eta) \geq 0$ for all cyclical perturbations $\eta \in \cC_\pi$} --- provided $L\eta^- < \infty$ so that the quantity $L(\eta)$ makes sense as an element of $\R \cup \{+\infty\}$ --- \emph{then $L(\eta) \geq 0$ for all finite perturbations $\eta \in \cF_\pi$} --- again, so that $L(\eta^-) < \infty$. Here, in our present context, the linear function $L$ is simply given by $\C$ the cost functional of \eqref{OT}.
\end{rem}

The game will be to show that (under suitable assumptions) \textbf{\textit{the cost functional $\C$ being non-negative over the radial cone $\cR_\pi$ is equivalent to it being non-negative over the set of finite perturbations $\cF_\pi$ or equivalently on the set of cyclical perturbations $\cC_\pi$}}. We will first show (under suitable assumptions) that $c$-cyclical monotonicity is necessary for optimality in Section~\ref{sec:nec-ot} (``\textit{If $\C$ is non-negative over $\cR_\pi$, then it is non-negative over $\cC_\pi$ (or equivalently over $\cF_\pi$)}'') then prove that it is sufficient in Section~\ref{sec:suf-ot} (``\textit{If $\C$ is non-negative over $\cC_\pi$ (or equivalently over $\cF_\pi$), then it is non-negative over $\cR_\pi$}''). In fact, we will show every element of the radial cone $\cR_\pi$ is an accumulation point of $\cF_\pi$ (see Proposition~\ref{prop:density} \textit{infra}). This last property is actually somehow completely agnostic to the cost of transport, but allows to retrieve the equivalence between optimality and $c$-cyclical monotonicity under appropriate continuity assumptions on the cost of transport $c$ (see Proposition~\ref{prop:sufficiency-cont} \textit{infra}).

\subsection{\textit{Optimality implies $c$-cyclical monotonicity} ($\Rightarrow$)}\label{sec:nec-ot}
We start by showing in our framework that optimality implies $c$-cyclical monotonicity. The most general result in this direction is that of \textcite{beiglbockOptimalBetterTransport2009} under the sole assumption that the cost of transport $c : X \times Y \to [0, \infty]$ be Borel-measurable. Here, we revisit the proof that optimality implies $c$-cyclical monotonicity under the slightly stronger assumption that $c$ is integrable with respect to $\mu \otimes \nu$:

\begin{thm}\label{thm:best-thm-necessary}
	Let $c : X \times Y \to \extendedR$ be Borel-measurable so that $c$ is (locally) integrable with respect to $\mu \otimes \nu$. If $\pi \in \Pi(\mu, \nu)$ is an optimal transport plan with $\C(\pi) < \infty$, then $\pi$ is $c$-cyclically monotone.
\end{thm}

We stress out that $c$ is allowed to assume $+\infty$ as value (and in fact we do not demand that $c$ be either non-negative nor bounded from below at this stage). In the following proposition, we first prove Theorem~\ref{thm:best-thm-necessary} in the special case where the marginal spaces are Euclidean, that is $X = \R^d$ and $Y = \R^m$ for $d, m \in \N$. This assumption will be removed afterwards.

\begin{prop}\label{prop:necessary-euclidean}
	Let $c : \R^d \times \R^m \to \extendedR$ be Borel-measurable so that $c$ is (locally) integrable with respect to $\mu \otimes \nu$. If $\pi \in \Pi(\mu, \nu)$ is an optimal transport plan such that $\C(\pi) < \infty$ then $\pi$ is $c$-cyclical monotone.
\end{prop}
\begin{proof}[Proof of Proposition~\ref{prop:necessary-euclidean}]
	Let us consider a (finite) sequence of points $(x_1, y_1)$, ...,$(x_n, y_n) \in \Gamma$ for a set $\Gamma \subset \R^d \times \R^m$ such that $\pi(\Gamma) = 1$.  \textit{The precise definition of the set $\Gamma$ is to be given at the end of the proof}. Here, recall that we let $y_{n+1} := y_1$. Let $\eps > 0$ and let $B_\eps^i \subset \R^d \times \R^m$ be the ball of radius $\eps$ centered at $(x_i, y_i)$. Again, let $B_\eps^{n+1} := B_\eps^1$. Up to refining the set $\Gamma$, we may (and will) assume without loss of generality that $\pi(B_\eps^i) > 0$ for all $\eps > 0$ and all $i =1, \dots, n$. Let us then define $\sigma_i^\eps$ to be the restriction of $\pi$ to the ball $B_\eps^i$ normalised so as to make it a probability measure, that is $\sigma_i^\eps = \pi(B_\eps^i)^{-1} \pi \restriction_{B_\eps^i}$. We then let $$\sigma_\eps := \sum_{ i = 1}^n \sigma_i^\eps.$$  We note that $\sigma_\eps \ll \pi$ and that the Radon-Nykodym derivative of $\sigma_\eps$ with respect to $\pi$ is uniformly bounded since $\pi(B_\eps^i) > 0$ for all $i = 1, \dots, n$. More precisely, we have $\Vert \frac{d \sigma_\eps}{d \pi} \Vert_{L^\infty(\pi)}^{-1} = \min_{i = 1, \dots, n} \pi(B_\eps^i)$. Now, we introduce $\xi_i^\eps := \pr_1^\sharp \sigma_i^\eps \otimes \pr_2^\sharp \sigma_{i+1}^\eps$ for all $i = 1, \dots, n$, where again we let $\sigma_{n+1}^\eps := \sigma_1^\eps$ and let us define $$\xi_\eps := \sum_{i =1}^n \xi_i^\eps.$$ By construction $\sigma_\eps$ and $\xi_\eps$ share common marginals. Therefore $\eta_\eps := \xi_\eps - \sigma_\eps \in \cR_\pi$ is an admissible perturbation at $\pi$  --- since $\pi + t(\xi_\eps - \sigma_\eps) \in \Pi(\mu, \nu)$ for all $0 \leq t < \Vert \frac{d \sigma_\eps}{d \pi} \Vert^{-1}_{L^\infty(\pi)}$  --- and by optimality (and finiteness) of $\pi$ it must be that $\C(\xi_\eps) \geq \C(\sigma_\eps)$. Let us show that $\C(\sigma_\eps) \to \sum_{i = 1}^n c(x_i, y_i)$ and $\C(\xi_\eps) \to \sum_{i = 1}^n c(x_i, y_{i+1})$ as $\eps \to 0$, so that the thesis of Proposition~\ref{prop:necessary-euclidean} will follow. 
	
	We now recall that by virtue of \textit{Lebesgue differentiation theorem} for $\pi$-a.s. all point $(x_0, y_0) \in \R^d \times \R^m$ one has
	\begin{equation}\label{eq:ldt}
		\frac{1}{\pi(B_\eps)}\int_{B_\eps} c(x, y) \, \pi(\dd x, \dd y) \cto{\eps \to 0} c(x_0, y_0)
	\end{equation}
	where $B_\eps$ is the ball of radius $\eps$ centred at $(x_0, y_0)$. This theorem is valid for any Radon measure on Euclidean spaces, see \textit{e.g.} \cite[Theorem 8.4.6]{benedettoIntegrationModernAnalysis2009}. This is veracious since $\pi$ has finite transport cost, meaning that $c \in L^1(\pi)$. Let then $\Gamma_0 \subset \R^d \times \R^m$ be the set of points $(x_0, y_0) \in \R^d \times \R^m$ such that \eqref{eq:ldt} holds — which has therefore full $\pi$-measure. Then — provided $(x_i, y_i) \in \Gamma_0$ for all $i = 1, \dots,n$ — we have
	\begin{multline*}
	    	\C(\sigma_i^\eps) = \int_{\R^d \times \R^m} c(x, y) \, \sigma_i^\eps(\dd x, \dd y) \\= \frac1{\pi(B_\eps^i)}\int_{B_\eps^i} c(x, y) \, \pi(\dd x, \dd y) \cto{\eps \to 0} c(x_i, y_i).
	\end{multline*}
    We have therefore obtain the sought-for claim that $$\C(\sigma_\eps) = \sum_{i = 1}^n \C(\sigma_i^\eps) \to \sum_{i =1}^n c(x_i, y_i)$$ as $\eps \to 0$. It remains to take care of the other limit. 
    
    We have
	\begin{align*}
			\C(\xi_i^\eps) &= \int_{\R^d \times \R^m} c(x, y) \,  \pr_1^\sharp \sigma_i^\eps \otimes \pr_2^\sharp \sigma_{i+1}^\eps(\dd x, \dd y) \\
			%&= \int_{\R^d \times \R^m} c(x, y) \,d  \left(\int_{\R^m} \sigma_i^\eps(x, d y')\right) \otimes \left(\int_{\R^d} \pi_{i+1}^\eps(\dd x', y)\right) \\
			&=  \int_{\R^d \times \R^m} c(x, y) \,d  \left(\frac{1}{\pi(B_\eps^i)}\int_{\R^m} \pi_{|B_{i}^\eps}(x, d y')\right) \otimes \left(\frac{1}{\pi(B_\eps^{i+1})}\int_{\R^d} \pi_{|B_\eps^{i+1}}(\dd x', y)\right) \\
			%&= \frac{1}{\pi(B_\eps^i) \times \pi(B_\eps^{i+1})} \int_{(\R^d \times \R^m) \times (\R^d \times \R^m)} c(x, y) \, d \pi_{|B_\eps^i}\otimes \pi _{|B_\eps^{i+1}}(x, y', x', y)\\
			&=\frac{1}{(\pi \otimes \pi)(B_\eps^i \times B_\eps^{i+1})} \int_{B_\eps^i \times B_\eps^{i+1}} c(x, y) \,  \pi \otimes \pi(\dd x, \dd y', \dd x', \dd y).
	\end{align*} 
We can then appeal again to Lebesgue differentiation theorem, applied this time to the probability measure $\pi \otimes \pi \in \cP((\R^d \times \R^m) \times (\R^d \times \R^m))$ and for the integrand $c$  seen as the function $c(x, y', x', y) := c(x, y)$. Note that this is righteous since $c$ is assumed to be (locally) integrable with respect to $\mu \otimes \nu$. For all $i = 1, \dots,n$, we have $(x_i, y_{i+1}) \in B_\eps^i \times B_\eps^{i+1}$ and letting $\eps \to 0$ yields that $\C(\xi_i^\eps) \to c(x_i, y_{i+1})$, and therefore the claim that $\C(\xi_\eps) = \sum_{i = 1}^n \C(\xi_i^\eps) \to \sum_{i = 1}^n c(x_i, y_{i+1})$. This terminates (in essence) the proof of Proposition~\ref{prop:necessary-euclidean}. 

To finish, let us discuss how to construct the set $\Gamma$ introduced at the beginning of this proof. Let $\Gamma_1 := \bigcap_{i = 1, 2}\pr_i(\Sigma)$ where $\Sigma \subset (\R^d \times \R^m) \times (\R^d \times \R^m)$ is the set of points for which Lebesgue differentiation theorem holds this time for $\pi \otimes \pi$ and where evidently we define $\pr_1(x, y, z, t) = (x, y)$ and $\pr_2(x, y, z, t) = (z, t)$ for all $(x, y, z, t) \in (\R^d \times \R^m) \times (\R^d \times \R^m)$. It is well-known that the sets $\pr_1(\Sigma)$ and $\pr_2(\Sigma)$ need not be Borel-measurable, but that they are certainly Lebesgue-measurable \cite[Theorem 2.12]{crauelRandomProbabilityMeasures2002}. It then suffices to consider $\Gamma := \Gamma_0 \cap \Gamma_1$ where $\Gamma_0$ was introduced above as the (Borel-measurable) set of points for which Lebesgue differentiation theorem holds  \eqref{eq:ldt} in the case of $\pi$. Note that $\Gamma$ is \textit{a priori} only Lebesgue-measurable. Nonetheless, as such there exists a Borel set $\Gamma' \subset \R^d \times \R^m$ and a $\pi$-null set $N \subset \R^d \times \R^m$ so that $\Gamma = \Gamma' \cup N$. Then $\Gamma'$ is the sought-for $c$-cyclically monotone Borel-measurable set over which $\pi$ concentrates.
	\end{proof}

	In the preceding proposition, we have assumed that the marginal spaces where Euclidean so as to appeal to the standard Lebesgue differentiation theorem, see \cite[Theorem 8.4.6]{benedettoIntegrationModernAnalysis2009}.  In the \textit{non}-Euclidean setting, this theorem is still much valid but requires additional assumptions either on the underlying space or on the probability measure under consideration. Typically, one demands that this measure be \textit{doubling}, see \textit{e.g.} \cite[Chapter 2.9]{federerGeometricMeasureTheory1996}. But this is completely inadequate, since there are no reasons for a (optimal) transport plan to be doubling — and in fact it is easy to build counterexamples. The gap between the Euclidean and non-Euclidean setting in the Lebesgue differentiation theorem boils down to the choice of the \textit{differentiation basis} — \textit{i.e.} whether or not we can consider ``nice" shrinking sets like balls or cubes. But for what matters to us, this is rather unimportant. More precisely, we can attach to each point a sequence of measurable sets specifically tailored to this point so that the Lebesgue differentiation theorem holds:

\begin{lem}\label{lem:ldt-gen}
	Let $\gamma \in \cP(Z)$ where $Z$ is a Polish space. Then, for $\gamma$-a.s all $z \in Z$, there exists a sequence $E_k(z)$ of Borel sets all containing $z$ such that $\gamma(E_k(z)) > 0$ for all $k \in \N$ and such that for any function $\phi : Z \to \R$ integrable with respect to $\gamma$, we have 
	\begin{equation}\label{eq:ldt-non-euclidean}
		\frac{1}{\gamma(E_k(z))} \int_{E_k(z)} \phi(w)  \gamma(\dd w) \cto{k \to \infty} \phi(z)
	\end{equation}
    for $\gamma$ almost-surely all $z \in Z$.
\end{lem}
The proof of the technical Lemma~\ref{lem:ldt-gen} is postponed to Appendix~\ref{app:ldt}. We can prove Theorem~\ref{thm:best-thm-necessary} in full generality using this lemma:

\begin{proof}[Proof of Theorem~\ref{thm:best-thm-necessary}]
	The proof is identical to that of Proposition~\ref{prop:necessary-euclidean}, but this time we need to care for which kind of neighbourhood we are restricting $\pi$ onto. Let $(x_1, y_1), \dots, (x_n, y_n) \in \Gamma$, where $\Gamma$ is a set such that $\pi(\Gamma) = 1$ that will be constructed again \textit{a posteriori}. For each $z \in X \times Y$, let $E_k(z)$ be the Borel sets of Lemma~\ref{lem:ldt-gen}. Let us then write $E_{k}^i = E_{k}(z_i)$ where $z_i := (x_i, y_i)$. Let then $\pi_{i}^k$ be the normalised restriction of $\pi$ to $E_{k}^i$ and $\sigma_k := \sum_{i = 1}^n \sigma_i^k$. Then, we let $\xi_i^k := \pr_1^\sharp \sigma_i^k\otimes \pr_2^\sharp \sigma_{i+1}^{k}$ for all $i = 1, \dots, n$ (where $\sigma_{n+1}^k := \sigma_{1}^k$) and let $\xi_k := \sum_{i = 1}^n \xi_i^k$. In virtue of Lemma~\ref{lem:ldt-gen}, just as in the proof Proposition~\ref{prop:necessary-euclidean}, we have $\C(\sigma_k) \to \sum_{i =1}^n c(x_i, y_i)$ as $k \to \infty$. Also, again using Fubini's theorem in the same spirit as in the proof of Proposition~\ref{prop:necessary-euclidean}, write
\begin{equation}
	\C(\xi_i^k) = \frac{1}{\pi \otimes \pi(E_k^i \times E_{k+1}^i)} \int_{E_k^i \times E_{k+1}^i} c(x, y) \, \pi \otimes \pi(\dd x, \dd y', \dd x', \dd y)
\end{equation}
Now, the thesis of Lemma~\ref{lem:ldt-gen} still holds here — we refer to Remark~\ref{rem:ldt-gen-tensor-product} in the Appendix. In particular, we have $\C(\xi_i^k) \to c(x_i, y_{i+1})$ for all $i = 1, \dots, n$ and therefore $\C(\xi_k) \to \sum_{i =1}^n c(x_i, y_{i+1})$. This finishes the (core of the) proof. The set $\Gamma$ is built exactly as in the proof of Proposition~\ref{prop:necessary-euclidean}, so we do not repeat the argument here.
\end{proof}

\begin{figure}
	\includegraphics[width = 0.7\textwidth]{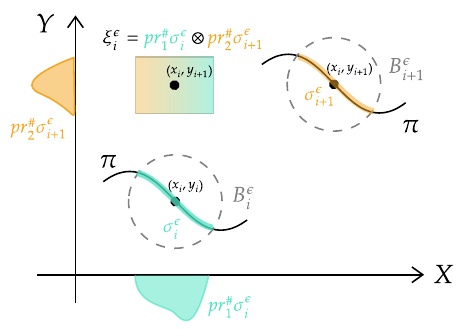}\label{fig:preuve1}
    \caption{Schematic explanation of the construction of the measures $\sigma_i^\eps$ and $\xi_i^\eps$ ($i=1, \dots,n$) in the proof of Proposition~\ref{prop:necessary-euclidean}.}
\end{figure}

\subsection{\textit{$c$-cyclical monotonicity implies optimality} ($\Leftarrow$)}\label{sec:suf-ot}
We will now discuss the reverse implication, namely that $c$-cyclical monotonicity is a sufficient criterion to yield optimality under suitable assumptions on the cost $c$. The idea is as follows: provided the $c$-cyclical monotonicity property holds, that is $\C(\eta) \geq 0$ for all $\eta \in \cC_\pi$ --- or equivalently for all $\eta \in \cF_\pi$ --- the game is to show that this implies that $\C(\eta) \geq 0$ for all $\eta \in \cR_\pi$, thus implying that $\pi$ is optimal according Proposition~\ref{prop:zero-order-cond}. In fact, we will first prove the following statement:
\begin{prop}\label{prop:density}
	For all $\pi \in \Pi(\mu, \nu)$, any element of the radial cone $\cR_\pi$ is an accumulation point of the set of finite perturbations $\cF_\pi$.
\end{prop}
Under the assumption that $\C$ is continuous -- \textit{i.e.} when $c \in C_0(X \times Y)$ -- then it follows immediately from the above proposition that $c$-cyclical monotonicity implies optimality:
\begin{prop}\label{prop:sufficiency-cont}
	Let $c : X \times Y \to \R$ be in $C_0(X \times Y)$, \textit{i.e.} continuous and vanishing at infinity (or yet that $c$ is continuous and the marginals $\mu$ and $\nu$ are compactly-supported). If $\pi \in \Pi(\mu, \nu)$ is $c$-cyclically monotone and $\C(\pi)<\infty$, then $\pi$ is optimal.
\end{prop}
\begin{proof}[Proof of Proposition~\ref{prop:sufficiency-cont}]
	We need that show that $\C(\eta) \geq 0$ for all $\eta \in \cR_\pi$. But for any admissible perturbation $\eta \in \cR_\pi$, according to Proposition~\ref{prop:density}, we can find $(\chi_n)_n \subset \cF_\pi$ such that $\chi_n \wto \eta$ in the limit $n \to \infty$. By $c$-cyclical monotonicity, we have $\C(\chi_n) \geq 0$ for all $n \in \N$. By the assumptions on the cost of transport $c$, the transport cost functional $\C : \cP(X \times Y) \to \R$ is vaguely continuous, which entails that $\C(\eta) = \lim_{n \to \infty} \C(\chi_n) \geq 0$ so that the thesis is proved.
\end{proof}
It is interesting to remark that the result of Proposition~\ref{prop:density} is in a sense completely disconnected from the notion of $c$-cyclical monotonicity \textit{per se}, since it has nothing to do with the cost of transport $c$. The $c-$cyclical monotonicity is certainly implied by it under continuity assumption on the cost of transport as shown above in Proposition~\ref{prop:sufficiency-cont}. We will show afterward how to obtain sufficiency of the $c$-cyclical monotonicity property for optimality without the stringent continuity assumptions on the cost of transport of Proposition~\ref{prop:sufficiency-cont}. First, we prove Proposition~\ref{prop:density} by constructing explicitly for each $\eta \in \cR_\pi$ a sequence of finite perturbations vaguely converging to $\eta$.

\subsubsection{Proof of Proposition~\ref{prop:density}} We fix an admissible perturbation $\eta \in \cR_\pi$. We remark that we may (and will) assume without loss of generality that $\eta^+$ and $\eta^-$ are normalised so that they are probability measures, that is $\int_{X\times Y} \eta^+ = \int_{X\times Y} \eta^- = 1$. By definition $\eta^+$ and $\eta^-$ share common marginals, that is $\pr_i^\sharp \eta^+ =\pr_i^\sharp \eta^- := m_i$ for $i = 1, 2$. We can then write $\eta^- = \eta^-_{y}(x) \otimes m_2(y)$ the disintegration of $\eta^-$ onto its second marginal $m_2$, and likewise $\eta^+ = m_1(x) \otimes \eta^+_x(y)$ the disintegration of $\eta^+$ onto its first marginal $m_1$. We are going to build samples $Z_0^-, \dots, Z_{n-1}^-$ — resp. $Z_0^+, \dots, Z_{n-1}^+$ — from the probability measure $\eta^-$ — resp. $\eta^+$ — such that \textit{\textbf{(i)}} the associated empirical measures $$\eta^-_n := \frac1n\sum_{i=0}^{n-1} \delta_{Z_i^-} \qquad \text{and} \qquad \eta^+_n := \frac1n\sum_{i=0}^{n-1} \delta_{Z_i^+}$$ vaguely converge to respectively $\eta^-$ and $\eta^+$ and \textit{\textbf{(ii)}} such that these empirical measures share common marginals. Then, $\eta_n := \eta^+_n - \eta^-_n$ will be a rightful finite perturbation at $\pi$, \textit{i.e.} an element of $\cF_\pi$, and $\eta_n \wto \eta$.

Let us now explain how to build such samples. First, assume that $Z_0^- := (x_0, y_0)$ is drawn according to the probability measure $\eta^-$. Now, draw $y_1 \in Y$ according to $\eta^+_{x_0}$ and let $Z_0^+ := (x_0, y_1)$. Then, to build $Z_1^-$, we draw $x_1 \in X$ according to $\eta^-_{y_1}$ and let $Z_1^- := (x_1, y_1)$. This procedure can be further continued, yielding two sequences of points $(Z_n^+)_{n \in \N}$  and $(Z_n^-)_{n \in \N}$ both inside the Cartesian product space $X \times Y$. The associated empirical measures $\eta^-_n$ and $\eta^+_n$ are such that their first marginals always coincide, i.e. $\pr_1^\sharp \eta_n^- = \pr_1^\sharp \eta_n^+$ --- or otherwise stated $\pr_1^\sharp \eta_n = 0$. Nevertheless, we have $\pr_2^\sharp \eta^+_n = \pr_2^\sharp \eta^-_n - \frac1n \delta_{y_0} + \frac1n \delta_{y_{n}}$ so that these two marginals do not \textit{a priori} coincide except in the eventuality that $y_0 = y_{n}$. We therefore need to correct this sampling procedure so as to ensure that the second marginals also coincide. 
 
 Before that, in all due respect, we need to put some rigour on the above construction. Let $\Omega := \Omega^- \times \Omega^+$ where $\Omega^- = \Omega^+ := X \times Y$, each of these spaces being equipped with the corresponding product $\sigma-$algebras. We let $\cZ_n := (Z_n^-, Z_n^+) \in \Omega_- \times \Omega_+$. Then, the sequence $(\cZ_n)_n$ is a time-homogeneous Markov chain with canonical paths space $\Omega^\N$ --- equipped again with its standard product $\sigma-$algebra. Let us remark that for all $n \in \N$ we have $\law(Z_n^-) = \eta^-$ and $\law(Z_n^-) = \eta^{-}$. Indeed, let $f : X \times Y \to \R$ be a test function, say continuous and bounded. We have
$$
 	\E[f(Z_1^-) | Z_{0}^- = (x_0, y_0)] = \int_{X \times Y} f(x, y) \,  \eta^+_{x_{0}}(\dd y) \otimes \eta^-_y(\dd x).
$$
 Therefore
$$
 	\E[f(Z_1^-)] = \int_{X \times Y }\left(\int_{X \times Y } f(x, y) \, \eta^+_{x_{0}}(\dd y) \otimes \eta^-_y(\dd x)\right) \eta^-(\dd x_{0},\dd  y_{0})
$$
 where we used that by construction $\law(Z_0^-) = \eta^-$. Let us then integrate over the variable $y_{0} \in Y$, so that
$$
 	\E[f(Z_1^+)] = \int_{X}\left(\int_{X \times Y } f(x, y) \,  \eta^+_{x_{0}}(\dd y) \otimes \eta^-_y(\dd x)\right) \, m_1(\dd x_0)
 	%\\ &= \int_{X \times Y} \left(\int_{X}\eta^-_{x_{0}}(y) m_1(x_0)\right) f(x, y) \eta^+_y(x) \\
 	%\label{:desint}&= \int_{X \times Y} f(x, y) \eta^+_y(x) m_2(y)\\
 	%\label{:reint}&= \int_{X \times Y} f(x, y) d\eta^-(x, y)
$$
 Then, let us integrate over $x_0 \in X$ and then over $y \in Y$
 \begin{align}
 	\notag \E[f(Z_1^-)]  &= \int_{X \times Y} f(x, y)\,  \eta^-_y(\dd x) \otimes \left(\int_{X} \eta^+_{x_{0}} \,  m_1(\dd x_0 )\right)(\dd y)
 	\\ 	\label{:desint}& = \int_{X \times Y} f(x, y) \, \eta^-_y(\dd x) \otimes m_2(\dd y)
 	\\ \label{:reint}&= \int_{X \times Y} f(x, y) \eta^-(\dd x, \dd y)
 \end{align}
 where we used disintegration of $\eta^+$ — resp. reintegration of $\eta^-$ — in \eqref{:desint} and \eqref{:reint}. Therefore, $\law(Z_1^-) = \eta^-$, and by an immediate induction $\law(Z_n^-) = \eta^-$ for all $n \in \N$. Using the same rationale, one shows that $\law(Z_n^+) = \eta^+$ for all $n \in \N$. 
 
We stress out that the empirical measures $\eta^-_n$ and $\eta^+_n$ are \textbf{random measures} $\eta^-_n = \eta^-_n(\omega)$ and $\eta^+_n = \eta^+_n(\omega)$ depending on the path $\omega \in \Omega^\N$. As noted above, we need to correct this sampling procedure so that the second marginals of these empirical measures coincide. Let us define 
$$
\chi^-_n := \eta^-_n + \frac1n \delta_{(\ox, \oy)} + \frac1n \delta_{(x_{n}, y_n)}, \qquad \chi^+_n := \eta^+_n + \frac1n\delta_{(\ox, y_{0})} + \frac1n \delta_{(x_{n}, \oy)} 
$$
where we choose any $(\ox, \oy) \in \supp(\eta^-)$ here. What we have done here is to add \textit{ad hoc} corrections to the empirical measures $\eta^-_n$ and $\eta^+_n$ so that now the resulting measure $\chi_n^-$ and $\chi^+_n$ verify that both of their marginals coincide, i.e. $\pr_i^\sharp \chi^-_n = \pr_i^\sharp \chi^+_n$ for all $n \in \N$ and $i = 1, 2$. Furthermore, we duly emphasise that $\supp(\chi_n^-) = \{(x_i, y_i) : i = 0, \dots, n\} \cup \{(\ox, \oy)\}$ is contained in $\supp(\eta^-)$, which is itself contained in the support of $\pi$ since $\eta \in \cR_\pi$. Therefore, letting $\chi_n := \chi^+_n - \chi^-_n$, we readily see that $\chi_n$ is a rightful finite perturbation at $\pi$, \textit{i.e.} $\chi_n \in \cF_\pi$. Now, let us rewrite $\chi^-_n = \eta^-_n + \varepsilon^-_n$ where $\eps^-_n := \frac1n \delta_{(\ox, \oy)} + \frac1n \delta_{(x_{n}, y_{n})}$ and likewise $\chi^+_n = \eta^+_n + \eps^+_n$ where $\eps^+_n := \frac1n\delta_{(\ox, y_{0})} + \frac1n \delta_{(x_{n}, \oy)}$. Since both $\eps^-_n$ and $\eps^+_n$ converge vaguely to zero as $n \to \infty$, if we can prove that the empirical measure $\eta^-_n$ (resp. $\eta^+_n$) converge vaguely to $\eta^-$ (resp. $\eta^+$) then the corrected measure $\chi^-_n$ (resp. $\chi^+_n$) will converge to $\eta^-$ (resp. $\eta^+$), henceforth $\chi_n$ will converge vaguely to $\eta$. This would imply the thesis of Proposition~\ref{prop:density} since $(\chi_n)_n$ would be a sequence of finite perturbations approaching $\eta$, thus proving that every element of the radial cone is an accumulation point of finite perturbations.

Nevertheless, some care should be given here, since it need \textit{not} be true that the empirical measures $\eta^-_n$ (resp. $\eta^+_n$) samples consistently the probability measures $\eta^-$ (resp. $\eta^+$) as $n \to \infty$. This is easily foreseen from the fact that the underlying Markov chain need not be \textit{ergodic}, as shown by the next example.

\begin{ex}
	Assume that $X = Y = [0,1]$. Let $S_\alpha(x) := x + \alpha \mod{1}$ be a shift, where $\alpha \in \R$. Let $\eta^-$ be the uniform distribution over the diagonal $\Delta := \{(x, x) : x \in [0,1]\}$ of the square $[0,1] \times [0,1]$ and $\eta^+$ be the uniform distribution over the shifted diagonal $\Delta_{\alpha} := \{(x, S_\alpha(x)):x \in [0,1]\}$. Otherwise stated $\eta^- := (Id, S_0)^\sharp \lambda$ and $\eta^+ := (Id, S_\alpha)^\sharp \lambda$ where $\lambda$ is the Lebesgue measure on $[0, 1]$. First, we remark that the initial condition $Z_0^- = (x_0, x_0)$ determines the values of the entire chain. More precisely, we have $Z_n^- = (S_\alpha^{(n)}x_0,S_\alpha^{(n)}x_0)$ and $Z_n^+ = (S_\alpha^{(n-1)} x_0, S_\alpha^{(n)} x_0 )$. If $\alpha = \frac{n}{m} \in \Q$  is rational, then $T_\alpha^{(m)} = Id$ and therefore a single chain will never sample entirely neither $\eta^+$ nor $\eta^-$. On the contrary, it is well-known that if $\alpha$ is irrational the underlying Markov chain is ergodic, and incidently $\eta^+$ and $\eta^-$ will be completely sampled from using a single chain. We note passing by that this example is closely related to a well-known counterexample in the literature to sufficiency of $c$-cyclical monotonocity to yield optimality, see \textit{e.g.} \cite[Example 1.2]{beiglbockOptimalBetterTransport2009}. 
\end{ex}
But the obstacle raised by the previous example can easily be bypassed for what matters to us — where all is fair game as long as we manage to sample from $\eta^+$ and $\eta^-$. Indeed, the idea to swerve away from this problem is simply to run \textit{multiple} and \textit{independent} copy of our  constructions. That is, we run several independent Markov chains $(\cZ_n^{j})_n$ for $j = 1, \dots, M$ where $\cZ_n^j := (Z_{n, j}^-, Z_{n, j}^+)$ each chain starting from independently drawn initial conditions $Z_{0, j}^- \sim \eta^-$ for all $j =1,\dots, M$. For each chain, let $\eta^-_{n,j}$ and $\eta^+_{n, j}$ be the empirical measures, $\chi^-_{n,j}$ and $\chi^+_{n,j}$ be the corrected empirical measures and $\eps^-_{n,j}$ (resp. $\eps^+_{n, j}$) the said corrections. We then consider the \textit{pooled} estimators
\begin{equation*}\label{eq:Lambdas-and-E}
	\underline{\eta}^-_{n, M} := \frac{1}{M} \sum_{j = 1}^M \eta^{-}_{n, j}, \quad  \underline{\chi}^-_{n,M} := \frac{1}{M} \sum_{j = 1}^M \chi^{-}_{n, j}, \quad \underline{\eps}^-_{n, M} := \frac1M \sum_{j = 1}^M \eps^-_{n, j}.
\end{equation*}
and their obvious counterparts $\underline{\eta}^+_{n, M}$, $\underline{\chi}^+_{n,M} $ and $\underline{\eps}^+_{n, M}$. We then let $\underline{\chi}_{n, M} := \underline{\chi}^+_{n,M}  - \underline{\chi}^-_{n,M}$. That is still an element of $\cF_\pi$. Now, for \textit{any} fixed $n \in \N$, it holds that the empirical measure $\underline{\eta}^-_{n, M}$ (resp. $\underline{\eta}^+_{n, M}$) vaguely converge to $\eta^-$ (resp. $\eta^+$) from shear independence of the chains in the limit $M \to \infty$ and since $\law(Z_{n,j}^-) = \eta^-$ and $\law(Z_{n,j}^+) = \eta^+$ for all $n \in \N$ and all $j \in \N$. On this matter, we yet remind the reader that both $\underline{\eta}^-_{n, M}$ and $\underline{\eta}^+_{n, M}$ are \textit{random} measures depending on $\omega \in \Omega^\N$, so that this convergence actually holds for almost-surely all $\omega$ (with respect to the underlying path measure). This is a classical result, see \textit{Varadarajan's theorem} \cite[Thm. 11.4.1]{dudleyRealAnalysisProbability2002}. We emphasise here that this result has essentially \textit{nothing} to do with the underlying Markovian structure of our construction, and is in essence just the law of large numbers.  Moreover, by the exact same rationale, it holds that $\underline{\eps}^-_{n, M}\wto e^-_n$ and $\underline{\eps}^+_{n, M}\wto e^+_n$ as $M \to \infty$ where $e_n^- = \frac1n \delta_{(\ox, \oy)} + \frac1n \eta^-$ and $e_n^+ = \frac1n \delta_{\overline{x}} \otimes m_2 + \frac1n m_1 \otimes \delta_{\overline{y}}$. 

To finish, let us rewrite $\underline{\chi}_{n, M} = \underline{\eta}_{n, M} - \underline{\eps}_{n, M}$ where we let $\underline{\eta}_{n, M} := \underline{\eta}^+_{n, M} - \underline{\eta}^-_{n, M}$ and likewise $\underline{\eps}_{n, M} := \underline{\eps}^+_{n, M} - \underline{\eps}^-_{n, M}$. The double-entry sequence $(\underline{\chi}_{n, M})_{n, M}$ verifies that $\underline{\chi}_{n, M} \wto \eta^+ - \eta^- -e^+_n + e^-_n$ as $M \to \infty$ and, since both $e^-_n$ and $e^+_n$ vanish vaguely in the limit $n \to \infty$, we obtain that the double (vague) limit $\lim_{n \to \infty} \lim_{M \to \infty} \underline{\chi}_{n, M}$ is equal to $\eta^+ - \eta^-$ , that is $\eta$. Finally, we remark that $\Vert \underline{\chi}_{n, M} \Vert_{TV} \leq 2(1+2/n) \leq 6$ for all $n, M \in \N$, where we recall that $\Vert \cdot \Vert_{TV}$ denotes the standard total variation norm.  We also recall that finite radii balls of $\cM(X \times Y)$ are metrisable for the vague topology. As in any metrisable space, we can perform a diagonal argument on the double-entry sequence $(\underline{\chi}_{n, M})_{n, M}$, so that there exists a sequence $(M_n)_n$ such that $\underline{\chi}_{n, M_n} \wto \eta$. This sequence being in $\cF_\pi$, we have proved that any admissible perturbation $\eta \in \cR_\pi$ can be approximated by finite perturbations, that is elements of $\cF_\pi$, thus yielding the thesis of Proposition~\ref{prop:density}.
\begin{rem}
	A small technical remark is to made here. We cannot work here with the \textit{narrow topology}, that is in duality with continuous and bounded functions $C_b(X \times Y)$. Indeed, neither the space $\cM(Z)$ of finite signed Radon measures is metrisable (except in the degenerate case where the base space $Z$ is finite) nor its unit ball (except when the base space $Z$ is compact, but in which case the narrow and vague topologies coincide). On this matter, see \textit{e.g.} \cite{varadarajanWeakConvergenceMeasures1958}.
\end{rem}

\subsubsection{Sufficiency of $c$-cyclical monotonicity without continuity}
It is shown above in Proposition~\ref{prop:sufficiency-cont} that $c$-cyclical monotonicity is a sufficient criterion for optimality under continuity assumption on the cost of transport. This was true in virtue of Proposition~\ref{prop:density}, in which we proved that every admissible perturbation $\eta \in \cR_\pi$ can be approximated, in the vague topology, by finite perturbations, \textit{i.e.} elements of the set $\cF_\pi$.  In this section, we would like to show how to remove these assumptions.

\begin{thm}\label{thm:sufficient-ot}
	Let $c : X \times Y \to [0, \infty]$ be Borel-measurable and assume that $c$ is finite $\mu \otimes \nu$-almost everywhere. Then, if $\pi \in \Pi(\mu, \nu)$ is $c$-cyclically monotone and $\C(\pi) < \infty$, then $\pi$ is optimal for \eqref{OT}.
\end{thm}

\begin{rem}
	We remark that the weakest theorem in this direction is that of \cite{beiglbockOptimalBetterTransport2009}, see \textit{Theorem 1.b} there. It states sufficiency of cyclical monotonicity under the assumption that $\{ c = \infty\} = F \cup N$ where $F$ is any closed set and $\mu \otimes \nu(N) = 0$. Our theorem above does not recover a part of this assumption, namely the one embodied by the closed set $F$. It may be possible to do that, although we have not attempted it.
\end{rem}

	Let us scrutinise more closely the construction built in the proof of Proposition~\ref{prop:density}. For the sake of  completeness, we recall some elementary fact here on the ergodic theorem of Birkhoff and how to apply it in the context of Markov chains, see \textit{e.g.} \cite{kallenbergFoundationsModernProbability2021a}. Let $T : \Omega^\N \to \Omega^\N$ be the shift operator $T(z_0, z_1, \dots) = (z_1, z_2, \dots)$, where we recall that $\Omega := \Omega^- \times \Omega^+$ where $\Omega^- = \Omega^+ = X \times Y$. According to what precedes, the Markov chain $(\cZ_n)_{n}$ admits as a stationary measure the probability measure $\Xi \in \cP(\Omega)$ defined as
	$$
	\Xi(B) := \int_{B}  \, \eta^-(\dd x^-,\dd y^-) \otimes \big(\delta_{x^-}(\dd x^+) \otimes \eta^+_{x^-}(\dd y^+)\big)
	$$
	for all Borel sets $B \in \Omega$. We let $\Pro_\Xi \in \cP(\Omega^\N)$ be the path probability measure of the chain with initial distribution $\Xi \in \cP(\Omega)$. By stationarity of $\Xi$ in the Markovian sense, the probability measure $\Pro_\Xi$ is left invariant by the shift operator  $T$, \textit{i.e.} $T^\sharp \Pro_\Xi = \Pro_\Xi$. The celebrated theorem due to Birkhoff states that for any measurable function $f : \Omega^\N \to \R$ which is integrable with respect to $\Pro_\Xi$, it holds that 
	$$
	\frac1n \sum_{k = 1}^n f \circ T^k(\omega) \to \E_{\Pro_\Xi}[f | \mathcal{I}](\omega)
	$$
	for $\Pro_\Xi$ almost-surely all $\omega \in \Omega^\N$ where $\E_{\Pro_\Xi}[f | \mathcal{I}]$ denotes the conditional expectation under $\Pro_\Xi$ of the function $f$ with respect to the $\sigma-$algebra consisting of all Borel sets left invariant by the shift, \textit{i.e.} $\mathcal{I} := \{ B \in \Omega^\N : T^{-1}(B) = B\}$. For what matters to us, we should also state that the same result applies in the case where $f$ is not necessarily integrable with respect to $\Pro_\Xi$ but where it is nevertheless (measurable and) \textit{non-negative}, in which case the convergence almost-surely still holds but the conditional expectation may assume $+\infty$ as value.  In particular, we can apply Birkhoff theorem to any function $f$ of the form $g - h$ where $g$ is measurable and non-negative and $h \in L^1(\Pro_\Xi)$. In particular, let us consider the function $f_c : \Omega^\N \to \R$ defined as $f_c(\omega) := c(x_0^+, y_0^+) - c(x_0^-, y_0^-)$ where we write $\omega = (z_0, \dots) \in \Omega^\N$ with $z_0 = (z_0^-, z_0^+) \in \Omega$ and where we further write $z_0^- = (x_0^-, y_0^-) \in \Omega^-$ and $z_0^+ = (x_0^+, y_0^+) \in \Omega^+$. Here, $g(\omega) = c(z_0^+)$ which is non-negative and $h(\omega) = c(z_0^-)$ which is integrable with respect to $\Pro_\Xi$, since
	$$
	\int_{\Omega_\N} h(\omega)  \Pro_\Xi(\dd \omega) = \int_{\Omega} c(z_0^-) \,  \Xi(\dd z_0^-, \dd z_0^+) = \int_{\Omega^-} c(z^-_0)\,  \eta^-(\dd z^-_0) < \infty
	$$
	where we successively used that the marginal (at ``time'' $n = 0$ ) of the path measure $\Pro_\Xi$ is (by definition) $\Xi$ and that the marginal (in space) of $\Xi$ with respect to $\Omega^-$ is $\eta^-$ --- and evidently that $\C(\eta^-) < \infty$ is finite by the assumption that $\C(\pi)$ is finite and the fact that $\eta^- \ll \pi$ and $\frac{d\eta^-}{d\pi} \in L^\infty(\pi)$. Let us then come back to the construction of the previous section and show how Birkhoff's theorem can be applied rightfully in order to conclude. 
	
	Let us first consider a single Markov chain $(\cZ_n)_n \in \Omega^\N$ where $\cZ_n = (Z_n^-, Z_n^+)$ and where $Z_n^- = (x_n, y_n)$ and $Z_n^+ = (x_n, y_{n+1})$. Beware \textit{again} that all these quantities are \textit{random} and depend intrinsically on a path $\omega \in \Omega^\N$, \textit{i.e.} $x_n = x_n(\omega)$ and $y_n = y_n(\omega)$. Applying what precedes, we have
	$$
	\frac{1}{n} \sum_{k= 0}^{n-1} [c(x_{k}, y_{k+1}) - c(x_{k}, y_{k})] \to \E_{\Pro_\Xi}[f_c | \mathcal{I}](\omega)
	$$
	for $\Pro_\Xi$ almost-surely all $\omega \in \Omega^\N$. Now, thanks to the $c$-cyclical monotonicity property which is assumed to hold, we have 
\begin{multline*}
   \frac1n \sum_{k = 0}^{n-1} [c(x_{k}, y_{k+1}) - c(x_{k}, y_{k})] \\+ \frac1nc(\ox, y_{0}) + \frac1n c(x_{n}, \oy) - \frac1nc(\ox, \oy) - \frac1n c(x_{n}, y_n) \geq0  
\end{multline*}
--- again for $\Pro_\Xi$ almost-surely all $\omega \in \Omega^\N$. By non-negativity\footnote{\label{fn:theonlyone} Of course, we could certainly work with a cost that is bounded from below using the known trick consisting in adding and removing this constant, or more generally with costs of transport which are bounded from below by a sum of functions which are integrable with respect to the marginals, see \textit{e.g.} \cite[Theorem 4.1]{villani_optimal_2009}} of the cost of transport $c$, this implies that
\begin{equation}\label{eq:eq}
\frac1n \sum_{k = 0}^{n-1} [c(x_{k}, y_{k+1}) - c(x_{k}, y_{k})]  \geq - \frac1nc(\ox, y_{0}) - \frac1n c(x_{n}, \oy).
\end{equation}

Let us assume for the moment that the cost of transport $c : X \times Y \to \R_+$ is uniformly bounded. Then, the right-hand side in the above inequality vanishes in the limit $n \to \infty$. This entails that $\E_{\Pro_\Xi}[f_c | \mathcal{I}] \geq 0$. But, taking again the expectation with respect to $\Pro_\Xi$, this implies that $\E_{\Pro_\Xi}[f_c] \geq 0$. But
\begin{align*}
\E_{\Pro_\Xi}[f_c] &:= \int_{\Omega^\N} f_c(\omega)  \Pro_\Xi(\dd\omega) \\
&= \int_{\Omega} [c(z_0^+) - c(z_0^-)]  \Xi(\dd z_0^-, \dd z_0^+) \\
&= \int_{\Omega^+} c(z^+_0) \eta^+(\dd z^+_0) -\int_{\Omega^-} c(z^-_0) \eta^-(\dd z^-_0) = \C(\eta) 
\end{align*}
where we used that the marginal of $\Xi$ with respect to $\Omega^-$ (resp. $\Omega^+$) is $\eta^-$ (resp. $\eta^+$). Thus, we obtain $\C(\eta) \geq 0$ as sought-for in the case where the cost of transport is uniformly bounded, leading to the sufficiency of $c$-cyclical monotonicity to yield optimality in this case. Let us now prove Theorem~\ref{thm:sufficient-ot} which states this result for less stringent assumption on the cost of transport.

	\begin{proof}[Proof of Theorem~\ref{thm:sufficient-ot}]
		The cost of transport $c$ is only assumed here to be finite $\mu \otimes \nu$ almost-surely. In this case, the right-hand side \eqref{eq:eq} does not necessarily vanish in the limit $n \to \infty$ and we cannot conclude as easily as in the case where $c$ is uniformly bounded. Here, the game is to select the point $(\overline{x}, \overline{y}) \in \supp(\pi)$ so that the two terms on the right-hand side vanish in the limit $n \to \infty$. Since $c$ is finite $\mu \otimes \nu$-almost everywhere, and thus also $m_1 \otimes m_2$-almost everywhere since $m_1 \ll \mu$ and $m_2 \ll \nu$, and since $\law(y_0) = m_2$, there exists a Borel set $A \subset X$ so that $m_1(A) = 1$ and such that for all $\ox \in A$ we have $\Pro_\Xi$ almost-surely that $c(\ox, y_0) < \infty$ and thus $\tfrac1n  c(\ox, y_0) \to 0$ as $n \to \infty$. Now, as for the second term, we remark that $\law(x_n) = \law(x_0) = m_1$ for all $n \in \N$. Therefore, we claim that there exists a set $B \subset Y$ so that $m_2(B) = 1$ and so that $\tfrac1n c(x_{n},\oy) \to 0$ \textit{in probability} for $\oy \in B$ — and thus $\Pro_\Xi$ almost-surely up to a subsequence. Indeed,
		$$
		\Pro_\Xi(\tfrac1n c(x_{n+1},\oy) > \eps) = \Pro_\Xi(c(x_0, \oy) > \eps n) \xrightarrow[n \to \infty]{} \Pro_{\Xi}(c(x_0, \oy) = \infty)
		$$
		Therefore — again by the finiteness assumption $\mu \otimes \nu$-almost everywhere of the cost of transport — there exists a set $B$ as claimed. Now, for both terms to vanish \textit{simultaneously}, it suffices to take any $(\ox, \oy) \in \supp(\eta^-) \cap (A \times B)$, which is non-empty since $\eta^-(A \times B) = 1$.
	\end{proof}

\section{Extension of cyclical monotonicity to the weak optimal transport} \label{sec:WOT}
\subsection{Setting and notations}
We use the strategy introduced above to extend the cyclical monotonicity to the weak optimal transport problem, namely:
\begin{equation}\label{WOT}\tag{WOT}
	\inf_{\pi \in \Pi(\mu, \nu)} \int_{X} C(x, \pi_x) \mu(\dd x)
\end{equation}
where $\mu \in \cP(X)$ and $\nu \in \cP(Y)$ are the two marginals and $C : X \times \cP(Y) \to \extendedR$ is throughout assumed to be convex with respect to its second variable, namely for every $x\in X$, the map $\rho \mapsto C(x, \rho)$ is convex. Here $(\pi_x)_{x \in X}$ denotes the disintegration of the transport plan $\pi$ onto its first marginal $\mu$ so that $\pi(\dd x, \dd y) = \mu(\dd x) \otimes \pi_x(\dd y)$. A \textit{sort of} cyclical monotonicity property for \eqref{WOT} has been introduced and coined as $C$-monotonicity in \cite{backhoff-veraguasExistenceDualityCyclical2019}. It reads as follows:
\begin{defi}\label{def:C-mon}
	A transport plan $\pi \in \Pi(\mu, \nu)$ is said to be \textit{$C-$monotone} if there exists a set $\Gamma \subset X$ with $\mu(\Gamma) = 1$ such that for any finite number of points $x_1, \dots, x_n \in \supp(\Gamma)$ and measures $m_1, \dots, m_n \in \cP(Y)$ such that $\sum_{i =1}^n m_i = \sum_{i = 1}^n \pi_{x_i}$ we have 
	\begin{equation}
		\sum_{i = 1}^n C(x_i, \pi_{x_i}) \leq \sum_{i = 1}^n C(x_i, m_i)
	\end{equation}
\end{defi}
It is known that this notion does \textit{not} extend rightfully the classical notion of cyclical monotonicity, since one does \textit{not} (always) recover it in the special where the weak optimal transport problem is actually the classical one, that is for the choice $C(x, \rho) := \int_{Y} c(x, y) \rho(\dd y)$. The following counterexample is to be found in \cite[Example 5.3]{backhoff-veraguasStabilityMartingaleOptimal2022}:
\begin{ex}\label{ex:c-mon-limitation}
	We assume that the underlying marginal spaces are $X = Y = [0,1]$ and that $C(x, \rho) = \int_{[0,1]} c(x, y)\rho(\dd y)$ where $c(x, y) := 1 - \mathds{1}_{\{x\}}(y)$. Let $\lambda$ be the Lebesgue measure on $[0,1]$. Then, the transport plan $\pi := \lambda \otimes \lambda \in \Pi(\lambda, \lambda)$ is easily seen to be $C$-monotone in the sense of Definition~\ref{def:C-mon}, but it is certainly not optimal for the optimal transport problem where $\mu = \nu = \lambda$ and where the cost of transport is precisely given by $c$.
\end{ex}

In contrast, we introduce in Definition~\ref{def:C-cm} below a more consistent extension of the $c$-cyclical monotonicity property for the weak optimal transport, in the sense that it is equivalent to the classical cyclical monotonicity as in Definition~\ref{def:c-cM} when the weak transport is actually the classical one. Before that, we discuss an \textit{ad hoc} notion of differentiability that will be used throughout this section.
\subsubsection{A notion of differentiability}\label{sec:diff}
Let us explain the notion of differentiability that will be used throughout this section. Here, for each fixed $x \in X$, the convex functional $\cP(Y) \ni \rho \to C(x, \rho)$ is canonically extended to the set $\cM(Y)$ of all signed and finite Radon measures by setting $C(x, \rho) = +\infty$ if $\rho \notin \cP(Y)$. We first let $\partial_\rho C(x, \overline{\rho})$ be the subdifferential at the point $\overline{\rho} \in \cP(Y)$ defined as
\begin{multline}\label{eq:subdiff-cont}
    \partial_\rho C(x, \overline{\rho}) := \Bigg\{ g \in C_0(Y) : C(x, \rho) \geq C(x, \overline{\rho}) \\+ \int_{Y} g(y)  [\rho - \overline{\rho}](\dd y), \, \forall \rho \in \cM(Y)\Bigg\}
\end{multline}
where we recall that $C_0(Y)$ denotes the set of real-valued continuous functions on $Y$ which vanish at infinity. The elements of the subdifferential are the so-called subgradients. For what matters to us, this definition is slightly too strong since we want to cope with subgradients which are \textit{not} necessarily continuous. To this end, we define an \textit{extended} subdifferential as follows. 

Let us write $\mathfrak{M}(Y)$ the set of Borel measurable functions $g : Y \to \R$. We say that $g \in \mathfrak{M}(Y)$ is an \textit{extended subgradient} of the function $\cP(Y) \ni \rho \to C(x, \rho)$ at the point $\overline{\rho} \in \cP(Y)$ if 
$$
C(x, \rho) \geq C(x, \overline{\rho}) + \int_{Y} g(y)[\rho - \overline{\rho}](\dd y)
$$
for all $\rho \in \cP(Y)$ \textbf{provided that $\int_Y |g| \overline{\rho} < \infty$ and that $\int_y |g| \rho < \infty$ as soon as $C(x, \rho) < \infty$.} The set of all $g$ verifying the above inequality and assumptions will be denoted $\mathfrak{D}_\rho C(x, \overline{\rho})$ and will be referred to as the \textit{extended subdifferential} and its elements as \textit{extended subgradients}. Obviously, this is a superset of the standard set of subgradients, that is $\partial_\rho C(x, \overline{\rho}) \subset \mathfrak{D}_\rho C(x, \overline{\rho})$. 

Now, we will say that $\rho \mapsto C(x, \rho)$ is \textit{differentiable} at $\overline{\rho} \in \cP(Y)$ if the extended subdifferential $\mathfrak{D}_\rho C(x, \overline{\rho})$ consists in a single element (up to additive constants). In this case, we will write (a canonical representation of) this element $\nabla_\rho C(x, \overline{\rho})$ and simply call it the\textit{ gradient.} It is a well-known fact of convex analysis that, assuming that $C(x, \overline{\rho}) < \infty$, it holds 
$$
\lim_{t \downarrow 0^+} \frac{C(x, (1-t)\overline{\rho} + t \xi) - C(x, \overline{\rho})}{t} = \sup_{g \in \partial_\rho C(x, \overline{\rho})}\int_{Y} g(y)[\xi - \overline{\rho}](\dd y)
$$
for all $\xi \in \cP(Y)$ such that $C(x, \xi) < \infty$. Now, evidently, by definition of extended subgradients, we also have:
$$
\lim_{t \downarrow 0^+} \frac{C(x, (1-t)\overline{\rho} + t \xi) - C(x, \overline{\rho})}{t} \geq  \sup_{g \in \mathfrak{D}_\rho C(x, \overline{\rho})}\int_{Y} g(y)[\xi - \overline{\rho}](\dd y)
$$
This entails that 
\begin{multline*}
\sup_{g \in \partial_\rho C(x, \overline{\rho})}\int_{Y} g(y)[\xi - \overline{\rho}](\dd y) \geq \\ \sup_{g \in \mathfrak{D}_\rho C(x, \overline{\rho})}\int_{Y} g(y)[\xi - \overline{\rho}](\dd y)\\ \geq \sup_{g \in \partial_\rho C(x, \overline{\rho})}\int_{Y} g(y)[\xi - \overline{\rho}](\dd y)
\end{multline*}
where, in the last inequality, we used that $\partial_\rho C(x, \rho) \subset \mathfrak{D}_\rho C(x, \rho)$. Therefore, we have equality of both \textit{suprema}. In particular, if $\cP(Y) \ni \rho \mapsto C(x, \rho)$ is differentiable in the aforementioned sense at the point $\overline{\rho} \in \cP(Y)$, then we have
\begin{equation}\label{eq:gat-diff-C-x}
\lim_{t \downarrow 0^+} \frac{C(x, (1-t)\overline{\rho} + t \xi) - C(x, \overline{\rho})}{t} = \int_{Y} \nabla_\rho C(x, \rho)(y) [\xi - \overline{\rho}](\dd y).
\end{equation}
provided that $C(x, \xi) < \infty$, where the gradient $\nabla_\rho C(x, \rho)$ is understood in the extended sense, \textit{i.e.} it is only considered measurable and thus may not be necessarily continuous.

\subsubsection{The $C$-cyclical monotonicity property}
We are now ready to introduce the \textit{$C$-cyclical monotonicity} property, which will serve as the analogue  of the $c$-cyclical monotonicity property of the classical optimal transport setting. This property reads as follows:
\begin{defi}\label{def:C-cm}
	A transport plan $\pi \in \Pi(\mu, \nu)$ is said to be \textit{$C$-cyclically monotone} if it is concentrated on a set $\Gamma \subset X \times Y$ so that for all $n \in \N$ and all $(x_1, y_1), \dots, (x_n, y_n) \in \Gamma$, it holds that
	\begin{equation}\label{c-mon-weak}
		\sum_{i = 1}^n \nabla_\rho C(x_i, \pi_{x_i})(y_i) \leq \sum_{i = 1}^n \nabla_\rho C(x_i, \pi_{x_i})(y_{i+1})
	\end{equation}
	where we let $y_{n+1} := y_1$ and where $Y \ni y \mapsto \nabla_\rho C(x, \pi_x)(y)$ is the gradient (see \textit{supra}) of the cost $C$ (provided it exists) with respect to its second argument evaluated at the point $(x, \pi_x) \in X \times \cP(Y)$. 
\end{defi}

\begin{ex}
	\begin{enumerate}
		\item In the special case where $C(x, \rho) = \int_{Y} c(x, y) \rho(\dd y)$, that is when the weak optimal transport problem boils down to the classical one, then the gradient of the cost of transport $C$ is simply the function $y \mapsto c(x, y)$ for all $x \in X$, and \eqref{c-mon-weak} is therefore seen to be the usual, standard $c$-cyclical monotonicity as of Definition~\ref{def:c-cM}.
		\item In the case of the \textit{barycentric quadratic cost}, namely $C(x, \rho) := \frac12 \Vert x - \int_{Y} z \rho(\dd z)\Vert^2$, then  $\nabla_\rho C(x, \rho)(y) = y^\intercal\left(\int_Y z \rho(\dd z) - x\right)$ for all $x \in X$ and $\rho \in \cP(Y)$ and the $C$-cyclical monotonicity property \eqref{c-mon-weak} reads
		\begin{equation*}
			\sum_{i = 1}^n y_i^\intercal\left(\int_z z \pi_{x_i}(\dd z) - x_i\right) \leq \sum_{i = 1}^n y_{i+1}^\intercal\left(\int_z z \pi_{x_i}(\dd z) - x_i\right).
		\end{equation*}
	\end{enumerate}

\end{ex}

Our main result in this section is that (under suitable assumptions) the $C$-cyclical monotonicity property as defined above in Definition~\ref{def:C-cm} is equivalent to optimality for the weak optimal transport problem \eqref{WOT}. The necessary part will be stated and proved below in Theorem~\ref{thm:necessary-wot} and the sufficient part in Theorem~\ref{thm:sufficient-wot}.

\subsection{The first-order optimality conditions for \eqref{WOT}} In this section, we follow the same strategy as in Section~\ref{sec:OT} where we derived the zeroth-order optimality condition for the classical OT problem. The main difference here is that the objective of the weak optimal transport problem 
$$\C(\pi) := \int_{X} C(x, \pi_x) \mu(\dd x)$$
is 
no longer linear. Nevertheless, under the assumption that $C$ be convex with respect to its second variable, the function $\C$ is convex, so that its minimum can be characterised using a \textit{first-order} optimality condition. We will assume from now on that $\C$ is extended from $\Pi(\mu, \nu)$ to the entire space of (finite and signed) Radon measures $\cM(X \times Y)$ by letting $\C = +\infty$ outside of the feasible set of transport plans. We start with the following lemma:

\begin{lem}\label{lem:wot-obj-cv-subdiff}
	Let $C : X \times \cP(Y) \to \extendedR$ be convex and differentiable with respect to its second variable. Then, the objective function $\C : \cM(X \times Y) \to \extendedR$ of the weak optimal transport problem is convex. Furthermore,  let $\pi \in \Pi(\mu, \nu)$ be such that $\C(\pi) < \infty$ and such that the following property holds 
	\begin{equation}\label{eq:int-assum}\tag{A$_\pi$}
		\C(\gamma) < \infty \implies \int_{X \times Y} |\nabla_\rho C(x, \pi_x)(y)| \gamma(\dd x, \dd y) < \infty,
	\end{equation}
	for all $\gamma \in \Pi(\mu, \nu)$. It is implicitly assumed in \eqref{eq:int-assum} that $X \times Y \ni (x,y)\mapsto \nabla_\rho C(x, \pi_x)(y)$ is jointly measurable. Then, for all $\gamma \in \Pi(\mu, \nu)$ such that $\C(\gamma)< \infty$, it holds that
	\begin{equation}\label{eq:algebraic-subgradient}
		\C(\gamma) \geq \C(\pi) + \iint_{X \times Y} \nabla_\rho C(x, \pi_x)(y) [\gamma - \pi](\dd x, \dd y).
	\end{equation}
	Furthermore, again  for all $\gamma \in \Pi(\mu, \nu)$ such that $\C(\gamma)< \infty$, we have
		\begin{multline}\label{eq:gat-dif-sub}
			\lim_{t \downarrow 0^+} \frac{\C((1-t)\pi + t \gamma) - \C(\pi)}{t} \\ = \int_{X \times Y} \nabla_\rho C(x, \pi_x)(y) [\gamma - \pi](\dd x, \dd y).
		\end{multline}
\end{lem}

Let us point out that the assumption~(A$_\pi$) (and the assumption~(A'$_\pi$) below) is only used to guarantee that the linearised functional is well-defined along all finite-cost competitors.

The inequality~\eqref{eq:algebraic-subgradient} of Lemma~\ref{lem:wot-obj-cv-subdiff} would read that the function $g$ is a \textit{subgradient} (in the usual sense, that is in duality with $C_0(X \times Y)$) of the cost functional $\C$ at the point $\pi$ if it was furthermore known that $g$ were \textit{at least} continuous. Nevertheless, there is no reason for this to be true, since the measurable map $X \ni x \mapsto \pi_x \in \cP(Y)$ is \textit{not} necessarily (vaguely) continuous. Instead, $g$ should be thought as an ``algebraic'' sort of subgradient. 

The technical condition \eqref{eq:int-assum} ensures that the integral on the right-hand side of \eqref{eq:algebraic-subgradient} makes sense as soon as the left-hand side is finite.  We will make use of this assumption throughout our result. We remark that, in the case where $C(x, \rho) = \int_{Y} c(x, y)\rho(\dd y)$, in which case $\nabla_\rho C(x, \rho)(y) = c(x, y)$, then this condition simply says that $\pi$ as finite transport cost, and is therefore always true since it is precisely assumed that $\pi$ has finite cost, that is $\int_{X \times Y} c \, d \pi < \infty$. It is also met for the barycentric squared cost if $\int_{Y} \Vert y \Vert^2 \nu(\dd y)< \infty$. Indeed, in this case $\nabla_\rho C(x, \rho)(y) = y^\intercal (\int_{Y} z \rho(\dd z) - x)$ and thus 

\begin{multline*}
\int_{X \times Y} |\nabla_\rho C(x, \pi_x)(y)|  \gamma(\dd x,\dd y) \leq \\ \sqrt{\int_{X \times Y} \Vert y \Vert^2  \gamma(\dd x,\dd y)} \sqrt{\int_{X \times Y} \left\Vert x - \int_{Y} z \pi_x(\dd z) \right\Vert^2\gamma(\dd x, \dd y)}
\end{multline*}
by the standard Cauchy-Schwarz inequality. Therefore, marginalising the integrals in the above inequality, we obtain
\begin{multline*}
\int_{X \times Y} \left\Vert x - \int_{Y} z \pi_x(\dd z) \right\Vert^2  \gamma(\dd x,\dd y) \\ = \int_{X} \left\Vert x - \int_{Y} z \pi_x(\dd z) \right\Vert^2 \mu(\dd x) = \C(\pi) < \infty
\end{multline*}
and 
$$
\int_{X \times Y} \Vert y \Vert^2  \gamma(\dd x,\dd y) = \int_{Y} \Vert y \Vert^2 \nu(\dd y) < \infty
$$
Thus $\int_{X \times Y} |\nabla_\rho C(x, \pi_x)(y)|  \gamma(\dd x, \dd y) < \infty$ is finite for all $\gamma \in \Pi(\mu, \nu)$ provided that $\nu$ has bounded second-moment. We note that this holds irrespective of whether or not $\C(\gamma)$ is finite, so that the assumption \eqref{eq:int-assum} is always verified. The same rationale can be applied to weak costs of transport of the form $C(x, \rho) := \theta(x - \int_{Y} z \rho(\dd z))$ where $\theta : \R^d \to \R$ is convex and differentiable  function that furthermore verifies $\Vert \nabla \theta \Vert^2 \leq \theta$ (up to a multiplicative constant). The assumption \eqref{eq:int-assum} is also seen to hold for the \textit{Marton cost}, \textit{i.e.} $C(x, \rho) = (\int_{Y} \mathds{1}_{\{x\}}(z) \rho(\dd z))^2$ — with underlying spaces $X = Y = [0,1]$. Altogether, this assumption is general enough to encompass important cases of applications of the weak optimal transport.

\begin{proof}
	Let $\pi, \gamma \in \Pi(\mu, \nu)$ with disintegrations $(\pi_x)_{x \in X}$ and $(\gamma_x)_{x \in X}$ with respect to the marginal $\mu$. Then, the disintegration of the convex combination $(1-t)\pi + t \gamma$ with respect to $\mu$ is simply given by the convex combination of the disintegrations, to wit $((1-t)\pi_x + t \gamma_x)_{x \in X}$. Therefore,
	\begin{align*}
		\C((1-t)\pi + t \gamma) &= \int_{X} C(x, (1-t)\pi_x + t \gamma_x) \mu(\dd x).
		\\& \leq (1-t) \underbrace{\int_X C(x, \pi_x) \mu(\dd x)}_{\C(\pi)} + t \underbrace{\int_{X} C(x, \gamma_x) \mu(\dd x)}_{\C(\gamma)}
	\end{align*}
	where we used solely the convexity of $C$ with respect to its second argument. This proves that $\C$ is convex. Now, we also have that
	\begin{equation*}
	 \C(\gamma) - \C(\pi) = \int_{X} (C(x, \gamma_x) - C(x, \pi_x))\mu(\dd x)
	\end{equation*}
	for all $\gamma, \pi \in \Pi(\mu, \nu)$. Here, we note that we used that $\C(\pi) < \infty$, so that the above difference makes sense. For all $\gamma \in \Pi(\mu, \nu)$, using the convexity and differentiability of $C$ with respect to its second variable, we have
	\begin{equation}\label{eq:beware-1}
		C(x, \gamma_x) \geq  C(x, \pi_x) + \int_{Y} \nabla_{\rho} C(x, \pi_x)(y) \, [\gamma_{x} - \pi_{x}](\dd y)
	\end{equation}
	We emphasise here that by the technical assumption \eqref{eq:int-assum} the integrals on the right-hand side make sense. 
    Then, integrating with respect to $\mu$, we obtain
	\begin{equation}\label{eq:beware-2}
		 \C(\gamma) \geq \C(\pi) + \int_{X \times Y}  \nabla_\rho C(x, \pi_x)(y) \, \underbrace{[\gamma_x - \pi_x](\dd y) \otimes \mu(\dd x)}_{=  [\gamma - \pi](\dd x,\dd y)}
	\end{equation}
	Here, we relied implicitly (again) on the technical assumption \eqref{eq:int-assum} so as to appeal to Fubini-Lebesgue theorem --- as we did also above in \eqref{eq:beware-1}. This proves \eqref{eq:algebraic-subgradient}. Now, let us show \eqref{eq:gat-dif-sub}. First, applying the inequality \eqref{eq:algebraic-subgradient} that we have just proved to $\gamma \leftarrow (1-t)\pi + t \gamma$, we have immediately that 
	$$
	\liminf_{t \downarrow 0} \frac{\C((1-t)\pi + t \gamma) - \C(\pi)}{t} \geq \int_{X \times Y} \nabla_\rho C(x, \pi_x)(y) [\gamma - \pi](\dd x, \dd y).
	$$
	Let us now show that the \textit{supremum} limit is smaller or equal to the right-hand side, so that \eqref{eq:gat-dif-sub} will follow. We have
	$$
	\frac{\C((1-t)\pi + t \gamma) - \C(\pi) }t= \int_{X} \phi_t(x) \mu(\dd x)
	$$
	where we define $\phi_t : X \to \R$ as 
	$$
	\phi_t(x) := \frac{C(x, \pi_x + t(\gamma_x - \pi_x)) - C(x, \pi_x)}{t}.
	$$
	The function $\phi_t$ is measurable. Furthermore, for each fixed $x \in X$, thanks to the convexity of $C$ with respect to its second variable, it is monotone in the variable $t \in [0,1]$. In particular, we have $\phi_t(x) \leq \phi_1(x)$ for all $x \in X$. We claim that $\phi_1$ is integrable with respect to $\mu$. Indeed, $\phi_1(x) = C(x, \gamma_x) - C(x, \pi_x)$, and both of these terms once integrating against $\mu$ yield respectively the weak costs $\C(\gamma)$ and $\C(\pi)$, which are assumed to be finite. Therefore, we can appeal to Fatou's lemma to obtain
	\begin{equation*}
    \begin{split}
      \limsup_{t \downarrow 0^+} \frac{\C((1-t)\pi + t \gamma) - \C(\pi) }t &\leq \int_{X} \limsup_{t \downarrow 0} \phi_t(x) \, \mu(\dd x)\\
      &= \int_{X} \lim_{t \downarrow 0} \phi_t(x) \, \mu(\dd x)
    \end{split}
\end{equation*}

	Here, the \textit{supremum limit} is a limit by convexity of $\cP(Y) \ni \rho \mapsto C(x, \rho)$. But now, we recall by differentiability of $C$ with respect to its second variable that $$\lim_{t \to 0} \phi_t(x) = \int_{Y} \nabla_\rho C(x, \pi_x)(y) [\gamma_x - \pi_x](\dd y)$$ as shown in \eqref{eq:gat-diff-C-x}. We emphasise that the above limit holds for $\mu$ almost-surely all $x \in X$ since the assumption $\C(\gamma) < \infty$ implies that $C(x, \gamma_x)$ is finite for $\mu$ almost-surely all $x \in X$, which ensures that \eqref{eq:gat-diff-C-x} holds. This terminates the proof of \eqref{eq:gat-dif-sub}.
\end{proof}

We will show the following first-order optimality conditions for the weak OT problem:

\begin{lem}\label{lem:first-order-opt-cond-wot}
	Let $C : X \times \cP(Y) \to \extendedR$ be such that for all $x \in X$ the map $\cP(Y) \ni \rho \mapsto C(x, \rho)$ is convex and differentiable. Assume $\pi \in \Pi(\mu, \nu)$ has finite cost $\C(\pi)$ and the property \eqref{eq:int-assum} holds. Then $\pi$ is optimal for the weak optimal transport problem if and only if
	\begin{equation}\label{eq:first-order-opt-cond-wot}
	\boxed{\int_{X \times Y} \nabla_\rho C(x, \pi_x)(y)\eta^+(\dd x, \dd y) \geq \int_{X \times Y} \nabla_\rho C(x, \pi_x)(y)\eta^-(\dd x,\dd y)}
	\end{equation}
	for all $\eta \in \cR_\pi$ provided that $\C(\pi + t \eta) < \infty$ for all $t > 0$ small enough. 
	\end{lem}
	 We recall that $\cR_\pi$ is the radial cone to feasible set $\Pi(\mu, \nu)$ at the point $\pi \in \Pi(\mu, \nu)$ as defined in \eqref{eq:R_pi}. We also recall that $\eta^+$ (resp. $\eta^-$) denotes the positive (resp. negative) part of the signed and finite Radon measure $\eta \in \cM(X \times Y)$. The inequality \eqref{eq:first-order-opt-cond-wot} is the first-order optimality condition analogue of the zeroth-order optimality condition of Proposition~\ref{prop:zero-order-cond}. Here, we need to go to the first order because the objective of the WOT problem is no longer linear — see also Remark~\ref{rem:first-order-opt-cond} below for a more general picture of the situation.
\begin{proof}[Proof of Lemma~\ref{lem:first-order-opt-cond-wot}]

We start with the sufficient part. Beware that throughout, we work under the technical assumption \eqref{eq:int-assum}. Let $\gamma \in \Pi(\mu,\nu)$ be such that $\C(\gamma)<\infty$. By Lemma~3, we have
$$
\C(\gamma) \geq \C(\pi) + \int_{X\times Y} \nabla_\rho C(x,\pi_x)(y)\,[\gamma-\pi](\dd x,\dd y).
$$
We now prove that the last integral is non-negative. Let us set $\eta := \gamma-\pi$, so that $\eta$ is an admissible perturbation at $\pi$, that is $\eta\in R_\pi$. Indeed, we have
$$
\pi+t\eta=(1-t)\pi+t\gamma\in\Pi(\mu,\nu) \qquad \text{for all } t \in [0,1].
$$
Moreover, by convexity of $\C$,
$$
\C(\pi+t\eta) = \C((1-t)\pi+t\gamma) \leq (1-t)\C(\pi)+t\C(\gamma) <\infty
$$
for all $t\in[0,1]$. Hence, we can rightfully appeal to the condition ~\eqref{eq:first-order-opt-cond-wot}. Writing $\eta=\eta^+-\eta^-$ for the Jordan decomposition of the signed measure $\eta$, the condition
\eqref{eq:first-order-opt-cond-wot} yields
$$
\int_{X\times Y} \nabla_\rho C(x,\pi_x)(y)\,\eta^+(\dd x,\dd y) \geq \int_{X\times Y} \nabla_\rho C(x,\pi_x)(y)\,\eta^-(\dd x,\dd y).
$$
Consequently,
$$
\int_{X\times Y} \nabla_\rho C(x,\pi_x)(y)\,[\gamma-\pi](\dd x,\dd y) = \int_{X\times Y} \nabla_\rho C(x,\pi_x)(y)\,\eta(\dd x,\dd y) \geq 0.
$$
Therefore $\C(\gamma)\geq \C(\pi)$. Since this holds for every $\gamma\in\Pi(\mu,\nu)$ with $\C(\gamma)<\infty$, the transport plan $\pi$ is optimal.

Let us now deal with the necessary part of the lemma. Let $\pi \in \Pi(\mu, \nu)$ be an optimal transport plan for the weak optimal transport problem that verifies the assumptions of the lemma. Given a perturbation $\eta \in \cR_\pi$, let us show that \eqref{eq:first-order-opt-cond-wot} holds. By optimality of $\pi$ -- and the fact that $\C(\pi) < \infty$ -- we have 
$\C(\pi + t\eta) - \C(\pi) \geq 0$ for all $t > 0$ small enough. Dividing by $t$ this inequality, letting $t \to 0$ and appealing to \eqref{eq:gat-dif-sub}, we obtain that
\begin{equation*}
	\int_{X \times Y} \nabla_\rho C(x, \pi_x) (y) \eta(\dd x,\dd y) \geq 0.
\end{equation*}
We then obtain the claim that \eqref{eq:first-order-opt-cond-wot} holds by decomposing $\eta$ into its positive and negative parts. 
\end{proof}

\begin{rem}\label{rem:first-order-opt-cond}
	As in Remark~\ref{rem:zeroth-order-opt-cond}, let us indicate that is again a much general situation to which the weak optimal transport problem is only a particular case.  Let $X$ be a normed vector space, $f : X \to \extendedR$ be a (\textit{differentiable}) convex function on $X$ and $C \subset X$ be a convex set. Let us consider the convex minimisation problem
	\begin{equation}
		\min_{x \in C} f(x).
	\end{equation}
	Then, $x \in C$ is optimal if and only if $\nabla f(x)[h] \geq 0$ for all $h \in R_C(x)$ — where  $R_C(x)$ is the radial cone defined in \eqref{eq:radial-cone}. Indeed, let us assume $x$ is optimal, and let $h \in R_C(x)$. Then, $f(x + th) - f(x) \geq 0$ for $t > 0$ small enough, so that standardly dividing by $t$ and letting it tend to zero yields the claim. The other direction goes as follows: since $f$ is convex, we have $f(x) + \nabla f(x)[y - x] \leq f(y)$ for all $x, y \in C$. Thus, if $\nabla f(x)[h] \geq 0$ for all $h \in R_C(x)$, it is then immediate that $f(x) \leq f(y)$ for all $y$ and thus $x$ is optimal. In the case where $f$ is not differentiable, then we can use \textit{subgradients}. Then, if for all $h \in R_C(x)$ there exists at least one $g \in \partial f(x)$ — where $\partial f(x)$ denotes the (non-empty) subdifferential of $f$ at $x$ -- such that $\scal{g}{h} \geq 0$, then $x$ is optimal. Indeed, for all $y \in C$, letting $h := y - x$, then $f(x) \leq f(x) + \scal{g}{h} \leq f(x + h) = f(y)$ such $x$ is optimal. In the other direction, assume that $x$ is optimal. Then, for all $h \in R_C(x)$ and $t$ small enough $f(x + th) \geq f(x)$. This entails that $f'(x; h) \geq 0$ for all $h \in R_C(x)$ where we let $f'(x; h)$ be the directional derivative at $f$ in the direction $h$ (since $f$ is convex, it exists for all $h$ in $\extendedR$). Then either $f'(x; h) = 0$ for all $h$, then $f$ is actually differentiable at $x$ and $\nabla f(x) = 0$, so the conclusion is trivial. Otherwise, there exists a direction $h_0 \in R_C(x)$ such that $f'(x; h_0) > 0$. Using that $f'(x, h_0) = \sup_{g \in \partial f(x)} \scal{g}{h_0}$ this means that there exists a subgradient $g_0 \in \partial f(x)$ such that $\scal{g_0}{h_0} \geq 0$. Therefore, the exact statement is: if $x$ is optimal, then for all $h \in R_C(x)$ there exists $g \in \partial f(x)$ such that $\scal{g}{h}\geq0$.
\end{rem}

\subsection{\textit{$C$-cyclical monotonicity characterises optimality}}
We have shown, in the case of the classical OT problem, that one can easily go from the validity of the zeroth-order optimality condition for the classical OT problem onto $\cR_\pi$ from its validity onto $\cF_\pi$ — yielding the \textit{necessary part} of $c$-cyclical monotonicity for optimality — either by mere density if the cost functional $\pi \mapsto \int_{X \times Y} c \, d \pi$ was continuous, see Proposition~\ref{prop:sufficiency-cont}, or under less stringent assumptions on the cost of transport $c$ using the ergodic theorem, \textit{cf.} Theorem~\ref{thm:sufficient-ot}. \textit{Vice-versa}, we went from the validity of the zeroth-order optimality condition from $\cF_\pi$ onto $\cR_\pi$ — yielding the \textit{sufficient part }of $c$-cyclical monotonicity for optimality — either by mere density again, appealing to Proposition~\ref{prop:density}, if again the cost functional is continuous, cf. Proposition~\ref{prop:sufficiency-cont}.

Now, to prove that $C$-cyclical monotonicity as defined in Definition~\ref{def:C-cm} is equivalent to optimality for the weak optimal transport problem \eqref{WOT} is  in some sense completely ``free", thanks to the work done previously for the classical optimal transport problem.

First, let us remark that $C$-cyclical monotonicity as in Definition~\ref{def:C-cm} is trivially equivalent to the veracity of the \textit{first-order optimality condition} \eqref{eq:first-order-opt-cond-wot} when evaluated at \textit{cyclical perturbations} $\eta \in \cC_\pi$. Just as in the classical optimal transport case, we can define \textit{finite optimality} here, which simply consists in the validity of first-order optimality condition \eqref{eq:first-order-opt-cond-wot} when evaluated at \textit{finite perturbation} $\eta \in \cF_\pi$. We emphasise that finite optimality and $C$-cyclical monotonicity are again equivalent for the weak optimal transport. Indeed, it suffices to apply Remark~\ref{rem:equiv-fin-opt-c-cm} with the linear map $L : \cM(X \times Y) \to \extendedR$ defined as
$$
L(\eta) := \int_{X \times Y} \nabla_\rho C(x, \pi_x)(y) \eta(\dd x, \dd y).
$$
--- provided that the map $X \times Y \ni (x,y) \mapsto \nabla_\rho C(x, \pi_x)(y)$ is Borel measurable, as will be assumed in our theorems. 

Second, we also know by Lemma~\ref{lem:first-order-opt-cond-wot} above that optimality in \eqref{WOT} is completely equivalent to the first-order optimality condition \eqref{eq:first-order-opt-cond-wot} evaluated for perturbations $\eta \in \cR_\pi$ in the radial cone. So, it suffices to prove that (under suitable assumptions) we can go from these two properties back-and-forth, just like we did in the case of the classical optimal transport problem. But whereas previously, we worked with the cost functional $$\pi \mapsto \int_{X \times Y} c(x,y)  \pi(\dd x, \dd y)$$ to be non-negative of the radial cone $\cR_\pi$, now instead of the cost functional, we work with $$\eta \mapsto \int_{X \times Y} \nabla_\rho C(x, \pi_x)(y)  \eta(\dd x, \dd y)$$
which is to be thought as the linearisation of the weak cost functional $\C$ at $\pi$. To prove the necessary and sufficient parts of $C$-cyclical for optimality, we just need to re-enact the former proofs of Theorem~\ref{thm:best-thm-necessary} and Theorem~\ref{thm:sufficient-ot} with the ``cost" function $c(x, y) := \nabla_\rho C(x, \pi_x)(y)$.

The necessary part reads as follows:
\begin{thm}\label{thm:necessary-wot}
	Assume that the map $$X \times \cP(Y) \times Y \ni (x, \rho, y) \mapsto \nabla_\rho C(x, \rho)(y) \in \extendedR$$ is Borel-measurable and (locally) integrable with respect to $\mu \otimes \nu$. Then, if an optimal transport plan $\pi \in \Pi(\mu, \nu)$ has finite cost, that is $\C(\pi) < \infty$, and that it  furthermore verifies the property \eqref{eq:int-assum} holds, then $\pi$ is $C$-cyclically monotone.  
\end{thm}
\begin{proof}[Proof of Theorem~\ref{thm:necessary-wot}]
This is direct from the (proof of) 
Theorem~\ref{thm:best-thm-necessary}. Indeed, we 
first remark that the map $x \mapsto \pi_x$ is 
measurable\footnote{This is understood in 
the sense that for all Borel set $B \in \cP(Y)$, 
the map $x \mapsto \pi_x(B) \in \R$ is Borel 
measurable. This implies by the monotone class theorem that $x \in X \mapsto \int_{X \times Y} f(x,y) \mu_x(\dd y)$ is Borel measurable for all (bounded from below) Borel functions $f$, see \cite[Section 5.3]{GradientFlows2008a}} by definition of the disintegration. 
Therefore, the map $X \times Y \ni (x, y) \mapsto 
\nabla_\rho C(x, \pi_x)(y)$ is also Borel-measurable. According to Lemma~\ref{lem:first-order-opt-cond-wot}, by optimality of $\pi$ -- 
and by the assumptions that $\C(\pi) < \infty$ 
and that the technical assumption \eqref{eq:int-assum} holds — we have $$\int_{X \times Y} c(x, y) 
\eta(\dd x, \dd y) \geq 0$$ for all $\eta \in 
\cR_\pi$ where we let $c(x, y) :=\nabla_\rho C(x, 
\pi_x)(y)$. To prove that $\pi$ is $C$-cyclically 
monotone, we want to prove that $\int_{X \times 
Y} c(x, y)  \eta(\dd x,\dd y) \geq 0$ for all 
$\eta \in \cC_\pi$, that is for all cyclical 
perturbations. To do so, it suffices to reiterate 
again the construction of the proof of 
Theorem~\ref{thm:best-thm-necessary} \textit{as-is}, where this time the ``cost" function $c : X 
\times Y \to \extendedR$ is the map defined 
above. Indeed, under the assumption of the above 
theorem, the function $c$ valued in the extended 
real line is Borel-measurable, which was the 
sole assumption of Theorem~\ref{thm:best-thm-necessary}.
\end{proof}

\begin{rem}
	This theorem should be compared to \cite[Theorem 5.3]{backhoff-veraguasExistenceDualityCyclical2019}, where it is proved that optimality implies $C$-monotonicity in under the assumptions that $C$ be jointly measurable, and convex and lower semicontinuous with respect to its second variable. Here, the sole underlying assumption on $C$ itself is that it be convex with respect to its second variable — but evidently assumptions on its gradient are needed. Of course, here, we remind the reader that we work with a different notion of cyclical monotonicity.
\end{rem}
As for the sufficient part, it is also as straightforward. First note that we do not have a clear counterpart of Proposition~\ref{prop:sufficiency-cont}. This is because there are no adequate assumptions making the function $(x, y) \mapsto \nabla_\rho C(x, \pi_x)(y)$ continuous (and vanishing at infinity). Indeed, the map $x \mapsto \pi_x$ need not be continuous under further structural assumption on $C$. For instance, for the classical OT problem, assuming that the optimal transport plan is of a Monge-type, meaning that $\pi_x = \delta_{T(x)}$ where $T$ is the associated transport map, then this requires that it be continuous, which need not be true in full generality. So, we move on directly to the counterpart of Theorem~\ref{thm:sufficient-ot}:
\begin{thm}\label{thm:sufficient-wot}
	Assume the map $$X \times \cP(Y) \times Y \ni (x, \rho, y) \mapsto \nabla_\rho C(x, \rho)(y)$$ is Borel-measurable, bounded from below, and finite $\mu\otimes\nu$-almost everywhere. Then, if $\pi \in \Pi(\mu, \nu)$ is $C$-cyclically monotone and has finite cost $\C(\pi) < \infty$, and provided property \eqref{eq:int-assum} holds, then it is optimal for the weak optimal transport problem \eqref{WOT}.
\end{thm}
\begin{proof}
	Again, as in the above proof for the necessary part, this is a direct application of the preceding proofs, more precisely that of Theorem~\ref{thm:sufficient-ot}. According to Lemma~\ref{lem:first-order-opt-cond-wot}, we need to show that 
    $$
    \int_{X \times Y} c(x, y) \eta(\dd x, \dd y) \geq 0
    $$
    for all $\eta \in \cR_\pi$ such that $\C(\pi + t \eta)$ is finite for all $t > 0$ small enough, where we let $c(x, y) :=  \nabla_\rho C(x,\pi_x)(y)$, provided the above inequality holds for finite perturbations, that this for all $\eta \in \cF_\pi$. Here, and for the same reasons as put forward in the proof of Theorem~\ref{thm:necessary-wot}, the function $c : X \times Y \to \extendedR$ is Borel-measurable and bounded from below under the assumptions of the theorem. It suffices to use the Markovian construction of the proof of Theorem~\ref{thm:sufficient-ot} with our specific cost $c$. 
\end{proof}

\subsection{Removing differentiability: the case of the barycentric cost}
In this section, we explore briefly if what precedes can be extended to the case where the weak cost of transport $C$ is not differentiable with respect to its second variable (but is nonetheless convex). In Section~\ref{sec:rem-dif}, we explain how one could extend the previous results without going into the details. In particular, we do not state any theorem. In Section~\ref{sec:bar-cost}, we rather specify more in detail what can be said in the case of the barycentric cost, which is not differentiable everywhere, as an illustration of the general arguments put forward in the first section. 
\subsubsection{Removing differentiability}\label{sec:rem-dif}
A natural question is if whether or not what precedes can be extended to the case where the weak cost of transport $C$ is not differentiable with respect to its second variable. Here, we recall that differentiability is understood in the sense introduced in Section~\ref{sec:diff}. Of course, because for all $x \in X$ the function $\cP(Y) \ni \rho \mapsto C(x, \rho)$ is convex, we can still work with (extended) subgradients. \textit{What is then the equivalent of Lemma~\ref{lem:first-order-opt-cond-wot} ?} To answer that, we first need to ask ourselves what is the equivalent of Lemma~\ref{lem:wot-obj-cv-subdiff}. In its proof, we remark that \eqref{eq:beware-1} shall be replaced with 
$$
C(x, \gamma_x) \geq C(x, \pi_x) + \int_{Y} g_x(y) [\gamma_x - \pi_x](\dd y)
$$
where $g_x : Y \to \R$ is a subgradient of $C(x, \cdot)$ at the point $\pi$. Of course, if $g_x$ is continuous and vanishing at infinity (or in fact simply bounded), then the integral $\int_{Y} g_x(y) \gamma_x(\dd y)$ always makes sense. This is not \textit{a priori} the case if $g_x$ is understood in the \textit{extended} sense, in which case $g_x$ is only assumed to be measurable. Furthermore, in the extended setting, for the above inequality to make any sort of sense, we have to assume that \textbf{$\int_{Y} |g_x(y)| \rho(\dd y) < \infty$ as soon as $C(x, \rho) < \infty$} --- at least for $\mu$ almost-surely all $x \in X$. Now, to obtain the equivalent of \eqref{eq:beware-2}, we have to integrate against $\mu$ the above inequality, yielding (formally)
$$
\C(\gamma) \geq \C(\pi) + \int_{X \times Y} g_x(y) [\gamma - \pi](\dd x,\dd y).
$$
Of course, here, we used formally the Fubini-Lebesgue theorem. For this to be readily justified, we need to add the following assumption, which is the analogue of the assumption \eqref{eq:int-assum}: for all $\pi \in \Pi(\mu, \nu)$ so that $\C(\pi) < \infty$, we have the property
	\begin{equation}\label{eq:int-assum_sub}\tag{A'$_\pi$}
	\C(\gamma) < \infty \implies \int_{X \times Y} |g_x(y)| \gamma(\dd x, \dd y) < \infty,
\end{equation}
for all $\gamma \in \Pi(\mu, \nu)$. Here, $g_x \in \mathfrak{D}_\rho C(x, \pi_x)$ is an extended subgradient at $\pi_x$ for all $x \in X$. We emphasise that here \textbf{it is implicitly assumed that the family of subgradients $(g_x)_{x \in X}$ under consideration in the statement of the assumption \eqref{eq:int-assum_sub} verifies that $X \times Y \ni (x, y) \mapsto g_x(y)$ is Borel-measurable}. We then obtain the following claim. \textit{If $\pi \in \Pi(\mu, \nu)$ has finite cost $\C(\pi) < \infty$ and if the property \eqref{eq:int-assum_sub} holds} --- which implicitly implies that for all transport plan $\gamma \in \Pi(\mu, \nu)$ such that $\C(\gamma) < \infty$ there exists a family $(g_x)_{x \in X}$ of extended subgradients $g_x \in \mathfrak{D}_\rho C(x, \pi_x)$ such that $(x, y) \mapsto g_x(y)$ is jointly measurable --- \textit{and if}
\begin{equation}\label{eq:cond-opt-not-diff}
 \int_{X \times Y} g_x(y) [\gamma - \pi](\dd x,\dd y) \geq 0,
\end{equation}
\textit{then $\pi$ must be optimal for the weak optimal transport problem \eqref{WOT}.} \textbf{Beware} that $g_x \equiv g_x^\gamma$ \textit{a priori} depends on $\gamma$. Conversely, let us assume that $\pi$ is an optimal transport plan for \eqref{WOT}. Then, provided that $\C(\gamma) < \infty$, we have that 
$$
\lim_{t \downarrow 0} \frac{\C(\pi + t(\gamma - \pi)) - \C(\pi)}{t} \geq 0
$$
Here, the limit exists as an element of the extended real-line $\extendedR$ since $\C$ is convex and by assumption $\C(\pi) < \infty$. By the same rationale as in the proof of Lemma~\ref{lem:wot-obj-cv-subdiff}, this entails that 
$$
\int_{X} \lim_{t \downarrow 0} \frac{C(x, \pi_x + t(\gamma_x - \pi_x)) - C(x, \pi_x)}{t} \mu(\dd x) \geq 0.
$$
Here, note that the limit of the integrand does exist as an element of the extended real-line $\extendedR$ since $C$ is convex with respect to its second variable \textit{and} that, from the fact that $\C(\pi) < \infty$, then $C(x, \pi_x)<\infty$ for $\mu$ almost-surely all $x \in X$. Therefore, using what has been devised in Section~\ref{sec:diff}, we have
\begin{equation}\label{eq:nasty-sub}
\int_{X} \left(\sup_{g \in \mathfrak{D}_\rho C(x, \pi_x)} \int_{Y} g_x(y) [\gamma_x - \pi_x](y)\right) \mu(\dd x) \geq 0.
\end{equation}

\subsubsection{The case of the barycentric cost}\label{sec:bar-cost}

We will not further explore the extension of what precedes in the case of \textit{any} (non-differentiable) convex weak cost of transport $C$. Rather, we now specify more in detail what happens in the case of the plain \textit{barycentric cost}. That is,  we have $X = Y = \R^d$ and $$C(x, \rho) = \left\Vert x - \int_{\R^d} z \rho(\dd z) \right\Vert.$$ 

 We have that the barycentric cost $C$ is differentiable with respect to its second variable at the point $(x, \rho) \in \R^d \times \cP(\R^d)$ if (and only if) $x \neq \int_{\R^d} z \rho(\dd z)$, and in this case $$\nabla_{\rho} C(x, \rho) (y)= -\frac{y^\top\left(x - \int_{\R^d} z \rho(\dd z)\right)}{\left\Vert x - \int_{\R^d} z \rho(\dd z)\right\Vert}$$ up to additive constants. Otherwise, if $x = \int_{\R^d} z \rho(\dd z)$, the subdifferential is not reduced to a single element. It is obvious that functions of the form $g(y) = u^\top y$ where $u \in \R^d$ is such that $\Vert u \Vert \leq 1$ are subgradients, although we stress out that there may exist subgradients which are not of this form\footnote{For instance, the function $g(y) = -\Vert y \Vert$ is an element of $\mathfrak{D}_\rho C(0, \delta_0)$.}.
    The supremum in \eqref{eq:nasty-sub} enjoys a totally explicit form. Indeed, it is attained for $g_x(y) := u_x^\top y$ where the vector $u_x \in \R^d$ is given by
    $$
        u_x = -\dfrac{x - \int_{\R^d} z \pi_x(\dd z)}{\Vert x - \int_{\R^d} z \pi_x(\dd z) \Vert}
    $$
    if $x \neq \int_{\R^d} z \pi_{x}(\dd z)$. Indeed, in this case, the extended subdifferential $\mathfrak{D}_\rho C(x, \pi_x)$ is reduced to a single function precisely stemming from the aforementioned vector. Otherwise, if $x = \int_{\R^d} z \pi_{x}(\dd z)$ and if $\int_{\R^d} z \gamma_x(\dd z) \neq x$, then we claim that 
    \begin{equation}\label{eq:u-x-second-case}
    u_x = \dfrac{\int_{\R^d} z \gamma_x(\dd z) - \int_{\R^d} z \pi_x(\dd z)}{\Vert \int_{\R^d} z \gamma_x(\dd z) - \int_{\R^d} z \pi_x(\dd z) \Vert}.
    \end{equation}

    Indeed, by definition of a subgradient $g : \R^d \to \R$, we have 
    \begin{equation}\label{eq:to-sature}
    \int_{\R^d} g(z) [\gamma_x - \pi_x](\dd z) \leq \left\Vert x - \int_{\R^d} z \gamma_x(\dd z) \right\Vert
    \end{equation}
    where it is used here that $C(x, \pi_x) = 0$ since we are in the case where $x = \int_{\R^d} z \pi_x(\dd z)$. Now, the choice \eqref{eq:u-x-second-case} precisely saturates the bound \eqref{eq:to-sature}, and thus is the optimiser in \eqref{eq:nasty-sub} in this context. Finally, the third and last case arises when $\int_{\R^d} z \gamma_x(\dd z) = \int_{\R^d} z \pi_x(\dd z) = x$. In this case, a subgradient $g$ must verify that 
    $$
    \int_{\R^d} g(z)[\gamma_x - \pi_x](\dd z) \leq 0.
    $$
    To sature this bound, it suffices to take $u_x = \overline{u}$ where $\overline{u} \in \R^d$ is any vector with $\Vert \overline{u} \Vert \leq 1$. Let us define $$A := \left\{x \in \R^d : x \neq \int_{\R^d} z \pi_x(\dd z)\right\}$$ and $$B_\gamma := \left\{x \in \R^d: \int_{\R^d} z \pi_x(\dd z) = x \neq \int_{\R^d} z \gamma_x(\dd z)\right\}.$$
    
	We remark that $A$ and $B_\gamma$ are Borel measurable sets since $x \mapsto \pi_x$ and $x \mapsto \gamma_x$ are measurable maps. Then, the inequality \eqref{eq:nasty-sub} reads 
	\begin{multline}
		-\int_{A} \left(\frac{x - \int_{\R^d} z \pi_x(\dd z)}{\Vert x - \int_{\R^d} z \pi_x(\dd z) \Vert}\right)^\intercal\left( \int_{\R^d} y [\gamma_x - \pi_x](\dd y) \right) \mu(\dd x) 
		\\ + \int_{B_\gamma} \left\Vert \int_{\R^d} y [\gamma_x - \pi_x](\dd y)\right\Vert \mu(\dd x) \geq 0
	\end{multline}
	or yet
	\begin{multline}\label{eq:eq33_duh}
		-\int_{A \times \R^d} \left(\frac{x - \int_{\R^d} z \pi_x(\dd z)}{\Vert x - \int_{\R^d} z \pi_x(\dd z) \Vert}\right)^\intercal y [\gamma - \pi](\dd x, \dd y) \geq\\ - \int_{B_\gamma} \left\Vert \int_{\R^d} z \gamma_x(\dd z) - x\right\Vert \mu(\dd x)
	\end{multline}
	We have the following theorem:
	\begin{thm}\label{thm:nec-cmon-bar}
		Let $\pi \in \Pi(\mu, \nu)$ be an optimal transport plan for the weak optimal transport \eqref{WOT} with barycentric cost, \textit{i.e.} $C(x, \rho)= \Vert x - \int_{\R^d} z \rho(\dd z)\Vert$. Let us also assume that $\int_{\R^d} \Vert y \Vert \nu(\dd y) < \infty$. Then, there exists a measurable set $\Gamma \subset \R^d \times \R^d$ such that $\pi(\Gamma) = 1$ and so that, for all $n \in \N$ and all points $(x_1, y_1), \dots, (x_n, y_n) \in \Gamma$, we have 
		\begin{equation}\label{eq:c-cmon-bar}
		\boxed{\sum_{i | x_i \neq b_i} \frac{(y_{i+1} - y_i)^\intercal(b_i - x_i)}{\Vert b_i - x_i \Vert} \geq -\sum_{i | x_i = b_i} \Vert y_{i+1} -y_{i} \Vert}
		\end{equation}
	where we let $b_i := \int_{\R^d} z\pi_{x_i}(\dd z)$ and where we let $y_{n+1} := y_1$. 
	\end{thm}
	This statement and the inequality \eqref{eq:c-cmon-bar} should be thought as an analogue of $c$-cyclical monotonicity for the weak optimal transport problem \eqref{WOT} in the specific case of the barycentric cost. The proof is in essence similar to that of Theorem~\ref{thm:best-thm-necessary}, although with some technical differences:
	\begin{proof}[Proof of Theorem~\ref{thm:nec-cmon-bar}]
		Let $(x_1, y_1), \dots, (x_n, y_n) \in \Gamma$ where $\Gamma \subset X \times Y$ is a Borel set so that $\pi(\Gamma) = 1$. This set is to defined at the end of the proof by successive refinements, similar to the proof of Theorem~\ref{thm:best-thm-necessary}. For each $i = 1, \dots,n$, let $B_i^\eps \subset \R^d \times \R^d$ the ``ball'' of radius $\eps > 0$ centred at $(x_i, y_i) \in \R^d \times \R^d$. More precisely, to simplify a bit the presentation, assume that $B_i^\eps:=B_\eps(x_i) \times B_\eps(y_i)$ where $B_\eps(x) := \{ z \in \R^d : \Vert z - x \Vert < \eps\}$. We may (and will) assume without loss of generality --- up to refining $\Gamma$ ---  that $\pi(B_i^\eps) > 0$ for all $\eps > 0$ and all $i =1,\dots, n$. Let $\sigma_i^\eps$ be the restriction of $\pi$ on $B_i^\eps$ normalised to a probability measure, so that $\sigma_i^\eps(A) := \pi(B_i^\eps)^{-1}\pi(A \cap B_i^\eps)$ for all Borel set $A \subset \R^d \times \R^d$, and let $\sigma^\eps := \sum_{i = 1}^n \sigma_i^\eps$. We then define $\xi_i^\eps  := \pr_1^\sharp \sigma_i^\eps \otimes \pr_2^\sharp \sigma^\eps_{i+1}$ and $\xi^\eps := \sum_{i =1}^n \xi_i^\eps$. Letting $\gamma^\eps := \pi + t(\xi^\eps - \sigma^\eps)$, then $\gamma^\eps \in \Pi(\mu, \nu)$ for all $\eps > 0$ and any $0 < t < t_\eps$ where $t_\eps := \Vert \frac{d \sigma^\eps}{d \pi} \Vert_{L^\infty(\pi)}$. We apply the inequality \eqref{eq:eq33_duh} to $\gamma \leftarrow \gamma^\eps$, which gives
	 \begin{multline}\label{eq:eq35}
	 	-\int_{A \times \R^d} \frac{(x - \int_{\R^d} z \pi_x(\dd z))^\top y}{\Vert x - \int_{\R^d} z \pi_x(\dd z) \Vert} [\gamma^\eps - \pi](\dd x, \dd y) \geq \\- \int_{B_{\gamma^\eps}} \left\Vert \int_{\R^d} z \gamma_x^\eps(\dd z) - x\right\Vert \mu(\dd x)
	 \end{multline}
	 where we recall that the measurable sets $A$ and $B_{\gamma^\eps}$ are defined by
     $$
     A = \left\{x \in \R^d : x \neq \int_{\R^d} z \pi_x(\dd z) \right\}
     $$
     and 
     \begin{equation}\label{eq:B}
     B_{\gamma^\eps} = \left\{ x \in A^c : x \neq \int_{\R^d} z \gamma^\eps(\dd z)\right\}.
     \end{equation}
     The above inequality \eqref{eq:eq35} further develops to
	 \begin{multline}\label{eq:eq36}
	 		-\int_{A \times \R^d} \frac{(x - \int_{\R^d} z \pi_x(\dd z))^\top y}{\Vert x - \int_{\R^d} z \pi_x(\dd z) \Vert} [\xi^\eps - \pi^\eps](\dd x, \dd y) \geq \\- \int_{B_{\gamma^\eps}} \left\Vert \int_{\R^d} z [\xi^\eps_x - \pi^\eps_x](\dd z) \right\Vert \mu(\dd x)
	 \end{multline} 
     Here, $(\xi^\eps_x)_{x \in \R^d}$ and $(\sigma^\eps_x)_{x \in \R^d}$ are the disintegrations of $\xi^\eps$ and $\sigma^\eps$ with respect to $\mu$. We have $\sigma^\eps_{x} = \sum_{i=1}^n \sigma^\eps_{i, x}$ where for all $i = 1, \dots, n$ the disintegration $(\sigma^\eps_{i, x})_{x \in X}$ of $\sigma^\eps_i$ with respect to $\mu$ reads for $x \in B_\eps(x_i)$
	 \begin{equation}
	 	\sigma^\eps_{i, x}(\dd y)= \frac{1}{\pi(B_i^\eps)}\pi_{x}\restriction_{B_\eps(y_i)}(\dd y)
	\end{equation}
	--- where $\pi_{x}\restriction_{B_\eps(y_i)}$ denotes the restriction of $\pi_x$ to the ball $B_\eps(y_i)$ --- and $\sigma^{\eps}_{i, x} = 0$ for $x \notin B_\eps(x_i)$ otherwise. Likewise, $\xi^\eps = \sum_{i=1}^n \xi^\eps_{i, x}$ where the disintegration $(\xi^\eps_{i, x})_{x \in X}$ of $\xi^\eps_i$ with respect to $\mu$ reads for $x \in B_\eps(x_i)$
	\begin{equation}
		\xi^\eps_{i, x}(\dd y) = \frac{1}{\pi(B_i^\eps)} \frac{\pi_x(B_\eps(y_i))}{\pi(B_{i+1}^\eps)} \int_{B_\eps(x_{i+1})} \pi\restriction_{\R^d \times B_\eps(y_{i+1})}(\dd x', y)
	\end{equation}
    and $\xi^\eps_{i, x} = 0$ for $x \notin B_\eps(x_i)$. To show that \eqref{eq:c-cmon-bar} holds, we are going to take the limit $\eps \to 0$ in the inequality \eqref{eq:eq36}.
    
	 First, we note that, by the standard \textit{Lebesgue differentiation theorem} and just like in the proof of Theorem~\ref{thm:best-thm-necessary}, it holds that --- up to refining the set $\Gamma$:
	 \begin{multline}\label{eq:ldt-simple}
    -\int_{A \times \R^d} \frac{(x - \int_{\R^d} z \pi_x(\dd z))^\top y}{\Vert x - \int_{\R^d} z \pi_x(\dd z) \Vert} [\xi^\eps - \pi^\eps](\dd x, \dd y) \\ \xrightarrow[\eps \to 0]{} \sum_{i | x_i \neq b_i} \frac{(y_{i+1} - y_i)^\top (b_i - x_i)}{\Vert b_i - x_i \Vert}
	 \end{multline}
     where we recall the notation $b_i := \int_{\R^d} z \pi_{x_i}(\dd z)$ for all $i = 1, \dots, n$. This is exactly the left-hand side of the sought-for inequality \eqref{eq:c-cmon-bar}. Now, it remains to prove that we can recover the right-hand side of \eqref{eq:c-cmon-bar} from the right-hand side of \eqref{eq:eq36} in the limit $\eps \to 0$.
     
     Let us first remark that the right-hand side rewrites as 
	 \begin{multline}
	     \int_{B_{\gamma^\eps}} \left\Vert \int_{\R^d} z [\xi^\eps_x - \pi^\eps_x](\dd z) \right\Vert \mu(\dd x) \\ = \sum_{i = 1}^n 	\int_{B_{\gamma^\eps} \cap B_\eps(x_i)} \left\Vert \int_{\R^d} z [\xi^\eps_x - \pi^\eps_x](\dd z) \right\Vert \mu(\dd x)
	 \end{multline}
    Indeed, here we used that $x \mapsto \xi_x^\eps$ and $x \mapsto \pi_x^\eps$ are supported onto the balls $B_\eps(x_i)$'s. Using the explicit formula for the disintegrations given above, we have 
     \begin{multline}
             \int_{B_{\gamma^\eps} \cap B_\eps(x_i)} \left\Vert \int_{\R^d} z [\xi^\eps_x - \pi^\eps_x](\dd z) \right\Vert \mu(\dd x) =\\
             \frac{1}{\pi(B_i^\eps)} \int_{B_i^\eps} \mathds{1}_{B_{\gamma^\eps}}(x) \Bigg\Vert \frac{1}{\pi_x(B_\eps(y_i))}\int_{B_\eps(y_i)}z \pi_x(\dd z) \\- \frac1{\pi(B_{i+1}^\eps)} \int_{B_{i+1}^\eps} y\pi(\dd x', \dd y)\Bigg \Vert \pi(\dd x, \dd y)
     \end{multline}
     For all $\eps > 0$ (small enough) we have 
     \begin{multline}
     \Vert y_{i+1} - y_i \Vert - 2\eps \\\leq \left\Vert \frac{1}{\pi_x(B_\eps(y_i))}\int_{B_\eps(y_i)}z \pi_x(\dd z) - \frac1{\pi(B_{i+1}^\eps)} \int_{B_{i+1}^\eps} y\pi(\dd x', \dd y)\right \Vert \\\leq \Vert y_{i+1} - y_i \Vert + 2\eps
     \end{multline}
     and therefore 
     \begin{multline}
         (\Vert y_{i+1} - y_i \Vert - 2\eps)\frac{1}{\pi(B_i^\eps)} \int_{B_i^\eps} \mathds{1}_{B_{\gamma^\eps}}(x) \pi(\dd x, \dd y)  \\
         \leq
         \int_{B_{\gamma^\eps} \cap B_\eps(x_i)} \left\Vert \int_{\R^d} z [\xi^\eps_x - \pi^\eps_x](\dd z) \right\Vert \mu(\dd x) \\
         \leq (\Vert y_{i+1} - y_i \Vert + 2\eps)\frac{1}{\pi(B_i^\eps)} \int_{B_i^\eps} \mathds{1}_{B_{\gamma^\eps}}(x) \pi(\dd x, \dd y)
     \end{multline}

We now claim that there exists $\overline{\eps} > 0$ small enough so that we have the set inclusion $B_{\gamma^\eps} \subset B_{\gamma^{\overline{\eps}}}$ for all $\eps < \overline{\eps}$. Indeed, let us remark that the set $B_{\gamma^\eps}$ rewrites as the set of points $x \in X$ so that $x = \int_{\R^d} z \pi_x(\dd z)$ and so that $\int_{\R^d} z \xi^\eps_x(\dd z) \neq \int_{\R^d} z \sigma^\eps_x(\dd z)$. This last condition rewrites, using the disintegrations, as 
\begin{equation}\label{eq:b_gamma_eps_}
    \frac{1}{\pi_x(B_\eps(y_i))} \int_{B_\eps(y_i)} z \pi_x(\dd z) \neq \frac{1}{\pi(B_{i+1}^\eps)} \int_{B_{i+1}^\eps} y \pi(\dd x', \dd y)
\end{equation}

     We note that, as $\eps \to 0$, this condition reads (formally for the moment) $y_i \neq y_{i+1}$. Of course, we may assume without loss of generality that $y_i \neq y_{i+1}$. Indeed, if it were that $y_i = y_{i+1}$, then
\begin{equation}
    \int_{B_{\gamma^\eps} \cap B_\eps(x_i)} \left\Vert \int_{\R^d} z [\xi^\eps_x - \pi^\eps_x](\dd z) \right\Vert \mu(\dd x) \leq \Vert y_{i+1} - y_i \Vert + 2\eps = 2\eps
\end{equation}
and thus, in the limit $\eps \to 0$, this integral does not contribute to the sum since it vanishes to zero. Therefore, only the case where $y_{i+1} \neq y_i$ is of importance. We note passing by that the previous result implies trivially that, in the limit $\eps \to 0$, the inequality \eqref{eq:eq36} implies
\begin{equation}
\sum_{i | x_i \neq b_i} \frac{(y_{i+1} - y_i)^\intercal(b_i - x_i)}{\Vert b_i - x_i \Vert} \geq -\sum_{i = 1}^n \Vert y_{i+1} -y_{i} \Vert.
\end{equation}
But the sought-for inequality \eqref{eq:c-cmon-bar} is finer since only the indices $i \in \{1, \dots, n\}$ where $x_i = b_i$ are consider in the sum on the right-hand side. Let us therefore assume that $y_i \neq y_{i+1}$ for all $i = 1, \dots,n$. If $x \in B_{\gamma^{\eps}}$, and using that $y_{i+1} \neq y_i$, then there exists a small enough $\overline{\eps} > 0$ such that \eqref{eq:b_gamma_eps_} holds for all $0 < \eps < \overline{\eps}$. Indeed, this follows from the obvious facts that 
$$
\left\Vert\frac{1}{\pi_x(B_\eps(y_i))} \int_{B_\eps(y_i)} z \pi_x(\dd z) - y_i \right\Vert \leq \eps
$$
and 
$$
\left\Vert \frac{1}{\pi(B_{i+1}^\eps)} \int_{B_{i+1}^\eps} y \pi(\dd x', \dd y) -y_{i+1} \right \Vert \leq \eps
$$
for all $\eps > 0$. Therefore, letting for instance $\overline{\eps} < \frac{\Vert y_{i+1} - y_i \Vert}{2}$, we have the set inclusion $B_{\gamma^{\eps}} \subset B_{\gamma^{\overline{\eps}}}$ for all $0 < \eps < \overline{\eps}$. 

Using this fact, we obtain 
     \begin{multline}
         \limsup_{\eps \to 0}\int_{B_{\gamma^\eps} \cap B_\eps(x_i)} \left\Vert \int_{\R^d} z [\xi^\eps_x - \pi^\eps_x](\dd z) \right\Vert \mu(\dd x) \\\leq \limsup_{\eps \to 0} \frac{\Vert y_{i+1} - y_i\Vert + 2\eps}{\pi(B_i^\eps)} \int_{B_i^\eps} \mathds{1}_{B_{\gamma^{\overline{\eps}}}}(x) \pi(\dd x, \dd y)%\\=\Vert y_{i+1} - y_i\Vert \mathds{1}_{B_{\gamma^{\overline{\eps}}}}(x_i)
     \end{multline}
Now, according to the Lebesgue differentiation theorem, we know that  there exists a set $E \subset \R^d$ so that $\mu(E) = 1$ and so that 
     $$
     \frac1{B_\eps(x)} \int_{B_\eps(x)} \mathds{1}_{B_{\gamma^{\overline{\eps}}}}(x') \mu(\dd x') \xrightarrow[\eps \to 0]{} \mathds{1}_{B_{\gamma^{\overline{\eps}}}}(x)
     $$
     for all $x \in E$. \textbf{If we could ensure that $x_i \in E$ for all $i = 1, \dots,n$}, then we would obtain 
     $$
      \limsup_{\eps \to 0}\int_{B_{\gamma^\eps} \cap B_\eps(x_i)} \left\Vert \int_{\R^d} z [\xi^\eps_x - \pi^\eps_x](\dd z) \right\Vert \mu(\dd x) \leq \Vert y_{i+1} - y_i\Vert \mathds{1}_{B_{\gamma^{\overline{\eps}}}}(x_i) 
     $$
     and thus the claimed inequality \eqref{eq:c-cmon-bar} would immediately follows. Let us discuss how to ensure that this happens. 
     
     First, beware that the measurable set $E \equiv E[n, \mathcal{X}, \overline{\eps}]$ actually depends on the function $\mathds{1}_{B_{\gamma^{\overline{\eps}}}}$ which itself depends on $n$, the family of points $\mathcal{X} := \{(x_i, y_i) : i = 1, \dots, n\}$ under consideration, and $\overline{\eps} > 0$. To simplify the presentation, let us write $B_{\gamma^{\overline{\eps}}} = B[n, \mathcal{X}, \overline{\eps}]$ as well. The idea is then: since any subspace of a metrisable separable space is itself separable, the support $\supp(\pi)$ of the transport plan $\pi$  is separable. Let us then consider $S \subset \supp(\pi)$ to be a dense and countable subset of $\supp(\pi)$. For each $n \in \N$, the set $\mathbb{X}_n = \{\mathcal{X}\}$ of all families $\mathcal{X} = \{(x_1, y_1), \dots, (x_n, y_n) \} \subset S$ of $n$ elements inside $S$ is therefore itself a countable set as a finite product of countable sets. Then, for each $\overline{\eps} > 0$, consider the set 
$$
E^{\overline{\eps}} := \bigcap_{n \in \N} \bigcap_{\mathcal{X} \in \mathbb{X}_n} E[n, \mathcal{X}, \overline{\eps}]
$$
The set $E^{\overline{\eps}}$ is measurable and $\mu(E^{\overline{\eps}}) = 1$. Then, we finally consider 
$$
E := \bigcap_{m \in \N} E^{1/m}
$$ 
The set $E$ is again measurable and $\mu(E) = 1$. We then consider $\Gamma := (E \times \R^d)\cap \Gamma_0$ where $\Gamma_0$ is the $\pi$-full measure set over which the convergence ensured by the Lebesgue differentiation theorem holds in \eqref{eq:ldt-simple}. We claim that the sought-for inequality \eqref{eq:c-cmon-bar} holds over $\Gamma$. This follows from the fact, for any $n \in \N$ and any family 
$$
\mathcal{X} = \{(x_1, y_1), \dots, (x_n, y_n) \} \subset \Gamma \subset (E \times \R^d),$$
(where \textit{a priori} $\mathcal{X} \not\subset S$) and for all $\eta > 0$, there exists a family $\mathcal{X}' \in \mathbb{X}_n$ which is $\eta$-close to $\mathcal{X}$, in the sense that, letting
$$
\mathcal{X}' = \{(x_1', y_1'), \dots, (x_n', y_n') \},
$$ 
we have $\Vert (x_i, y_i) - (x_i', y_i')\Vert < \eta$ for all $i = 1, \dots, n$. Now, let us consider $\eta = \eps/2$. Then, by similar arguments to that of preceding, we have that $B[n, \mathcal{X}, \eps] \subset B[n, \mathcal{X}', \eps/2]$. In particular, selecting $\overline{\eps} = 1/m$ for $m \in \N$ big enough, we obtain
 \begin{equation}
\frac1{\pi(B_i^\eps)} \int_{B_i^\eps} \mathds{1}_{B[n, \mathcal{X}, \overline{\eps}]}(x) \pi(\dd x, \dd y) \leq \frac1{\pi(B_i^\eps)} \int_{B_i^\eps} \mathds{1}_{B[n, \mathcal{X}', \overline{\eps}/2]}(x) \pi(\dd x, \dd y).
\end{equation}
Finally, the right-hand side converges to $\mathds{1}_{B[n, \mathcal{X}', \overline{\eps}/2]}(x_i) \Vert y_{i+1} - y_i\Vert$ since $x_i \in E$. This terminates the proof.

	\end{proof}

    The inequality \eqref{eq:c-cmon-bar} is to be thought as the analogue to the $c$-cyclical monotonicity property in the special case of the barycentric cost. Theorem~\ref{thm:nec-cmon-bar} says if $\pi$ is optimal, then it must verify this analogous property. A natural question is to go the other way around, that is to consider the following question: \textit{Is this true that, if $\pi$ verifies \eqref{eq:c-cmon-bar}, then $\pi$ is optimal for the weak optimal transport with barycentric cost ?} It is rather straightforward that:
    \begin{prop}
        Let $\pi \in \Pi(\mu, \nu)$ has finite, \textit{i.e.} $\C(\pi) < \infty$. Assume that $\mu$ almost-surely all $x \in \R^d$, we have $x \neq \int_{\R^d} z \pi_x(\dd z)$. Also assume that $\int_{\R^d} \Vert y \Vert \nu(\dd y) < \infty$. Then, if there exists $\Gamma \subset \R^d \times \R^d$ so that $\pi(\Gamma) = 1$ and so that for all $n \geq 1$ and all $(x_1, y_1), \dots, (x_n, y_n) \in \Gamma$, we have
        \begin{equation}\label{eq:c-cmon-bar-diff}
        \sum_{i=1}^n \frac{(y_{i+1} - y_i)^\top (b_i - x_i)}{\Vert b_i - x_i\Vert} \geq 0
        \end{equation}
        where $b_i = \int_{\R^d} z \pi_{x_i}(\dd z)$, then $\pi$ is optimal for the weak optimal transport problem \eqref{WOT} with barycentric cost.
    \end{prop}
    \begin{proof}
        Because for $\mu$ almost-surely all $x \in \R^d$, we have $x \neq \int_{\R^d} z \pi_{x}(\dd z)$, then the weak cost of transport $C$ is differentiable for $\mu$ almost-surely all $x$ with respect to its second variable and its gradient reads 
        $$
        \nabla_\rho C(x, \pi_x)(y) = \dfrac{(\int_{\R^d} z \pi_x(\dd z) - x)^\top y}{\Vert \int_{\R^d} z \pi_x(\dd z) - x \Vert } =: c(x, y)
        $$
        To show that $\pi$ is optimal, we must show that the first-order optimality condition holds, that is
        $$
        \int_{\R^d \times \R^d} c(x,y) [\gamma - \pi](\dd x, \dd y) \geq 0
        $$
        for all $\gamma \in \Pi(\mu, \nu)$ so that $\C(\gamma)<\infty$, \textbf{provided it holds for cyclical perturbations $\eta \in \cC_\pi$ --- as is the essence of \eqref{eq:c-cmon-bar-diff}.} This will follow by applying (the proof of) Theorem~\ref{thm:sufficient-wot}. We therefore only need to check that the ``cost'' $c$ defined above verifies the assumptions of this theorem. This function is real-valued and clearly Borel measurable --- by measurability of $x \mapsto \pi_x$, as already discussed. A bit of discussion is to be made regarding the range of values that the function $c$ above takes. We stated Theorem~\ref{thm:sufficient-wot} under the assumption that $c$ be non-negative. But this was mainly for convenience, and the theorem accommodates to slightly weaker assumptions, as already discussed above --- see the \textit{footnote} at p.~\pageref{fn:theonlyone}. In particular, we have 
        $$
        c(x, y) \geq - \Vert y \Vert := g(y)
        $$
        and $g$ is integrable with respect to $\nu$ since $\nu$ is assumed to have finite first-moment. This terminates the proof.
    \end{proof}
    In the case where there are points $x \in X$ so that $x = \int_{\R^d} z \pi_x(\dd z)$ (with positive $\mu$-measure) the situation is slighty more involved, since to obtain optimality we need to be able to select a family of (extended) subgradients so that \eqref{eq:cond-opt-not-diff} holds. Altogether, there is \textit{a priori} no reasons to obtain \eqref{eq:c-cmon-bar} by some limiting arguments as a sufficient criterion for optimality (which was obtained using a specific kind of subgradients, \textit{i.e.} the one that realises the \textit{supremum} in \eqref{eq:nasty-sub}). We could certainly explore more in detail this approach, but since this section is mainly intended to explain briefly what could be said in the non-differentiable case without going to much in details, we do not comment further on this point.

\subsection{A comment on the entropy-regularised problem}

When the weak cost of transport $C : X \times \cP(Y) \to \R_+$ is (strictly) convex with respect to its second-argument, this may push the (unique) solution of \eqref{WOT} inside the relative interior of $\Pi(\mu, \nu)$ compared to the \eqref{OT} problem whose solution lives on the boundary of the constraint set (in fact, on extreme points).  \textit{In this case, the defining inequality in the $C$-cyclical monotonicity introduced in Definition~\ref{def:C-cm} becomes an equality}. Indeed, the relative interior of the set of transport plan is made of those plans with full support, \textit{i.e.} $\supp(\pi) = X \times Y$. This means that if $\eta \in \cR_\pi$ is an admissible perturbation, then so is its opposite $-\eta$. This is for instance the case for the \textit{entropy-regularised optimal transport} problem, which we recall to be an instance of the weak optimal transport problem with 
$$
C(x, \rho) := \eps \int_{Y} \ln \rho \, d \rho + \int_Y c(x, y) \rho(\dd y)
$$
where $\eps > 0$ is the regularisation parameter. The weak cost of transport defined as such is differentiable in its second variable and its gradient (in the sense of Section~\ref{sec:diff}) reads
$$\nabla_\rho C(x, \rho)(y) = c(x, y) + \eps \ln \rho(y) + \eps.$$ 
As it is well-known, the unique solution to the entropic optimal transport lives in the relative interior of $\Pi(\mu, \nu)$. But then, we obtain that for all $(x_1, y_1), \dots, (x_n, y_n) \in \Gamma$ where $\Gamma \subset X \times Y$ is some set over which $\pi$ concentrates, we have 
\begin{equation}\label{eq:c-cm-ent}
	\boxed{\sum_{i = 1}^n c(x_i, y_{i}) + \eps \sum_{i = 1}^n \ln \pi_{x_i}(y_i) = \sum_{i = 1}^n c(x_i, y_{i+1}) + \eps \sum_{i = 1}^n \ln \pi_{x_i}(y_{i+1}) }
\end{equation}
	where as usual $y_{n+1} := y_1$. \textbf{\textit{The equality \eqref{eq:c-cm-ent} is then the analogue of the cyclical monotonicity property for the entropic optimal transport problem}}. This equality is not at all shocking, since using the well-known duality theory for the entropic optimal transport problem, we know that the optimiser $\pi^*$ is unique and given by the Gibbs measure $$\pi^*(x, y) = \exp\left(-\frac1\eps(c(x, y) - f(x) - g(y))\right)$$ for some entropic Kantorovich potentials $f : X \to \R$ and $g : Y \to \R$. In particular, the equality \eqref{eq:c-cm-ent} is in fact immediately verified. 

\appendix

\section{Proof of Lemma~\ref{lem:ldt-gen}}\label{app:ldt}
In this appendix, we prove Lemma~\ref{lem:ldt-gen} whose statement is repeated below for convenience of the readers. It is not unlikely that this result is known and common knowledge to some audiences, but as it is we were not able to find any reference for it, so that we prove it here.
\begin{lem}\label{lem:ldt-gen-repetita}
	Let $\gamma \in \cP(Z)$ where $Z$ is a Polish space (or in fact only second countable) equipped with its Borel $\sigma$-algebra. Then, for all $z \in Z$, there exists a sequence $(E_k(z))_{k \in \N}$ of Borel sets — all containing $z$ — so that $\gamma(E_k(z)) > 0$ for $\gamma$ almost-surely all $z \in Z$ and all $k \in \N$ and so that for any measurable function $\phi : Z \to \R$ which is integrable with respect to $\gamma$, it holds that 
	\begin{equation}
		\frac{1}{\gamma(E_k(z))} \int_{E_k(z)} \phi(w)  \gamma(\dd w) \cto{k \to \infty} \phi(z)
	\end{equation}
	for $\gamma$-a.s. all point $z \in Z$.
\end{lem}

\begin{proof}[Proof of Lemma~\ref{lem:ldt-gen}]
	As a Polish space, $Z$ is second countable, meaning that its topology has a countable base  $(U_k)_{k \in \N}$. Let us define $\cF_k := \sigma(U_1, \dots, U_k)$ the $\sigma-$algebra generated by $U_1, \dots, U_k$, so that $\mathcal{B}(Z) = \sigma( \bigcup_{k \geq 1} \cF_k)$ where $\mathcal{B}(Z)$ is the Borel $\sigma$-algebra on $Z$. Since $\cF_k$ is finitely generated by $k$ elements, it has at most $2^k$ atoms $e_1, e_2, \dots, e_{n_k} \in \cF_k$. Then, let us define $E_k(z)$ for all $z \in Z$ to be the atom in $\cF_k$ that contains $z$. We remark that for $\gamma$ almost-surely all $z \in Z$, we have $\gamma(E_k(z)) > 0$. Indeed, if $A_k := \{z \in Z : \gamma(E_k(z)) = 0\}$, then $A_k = \cup_{i \in I_k} e_i$ where $I_k \subset \{1, \dots, n_k\}$ is the set of indices so that $i \in I_k$ if $\gamma(e_i) = 0$. Since the atoms are disjoint by definition, we have $\gamma(A_k) = \sum_{i \in I_k} \gamma(e_i) = 0$. By considering the measurable set $A = \cup_{k \in \N} A_k \in \mathcal{B}(Z)$, we obtain that for $\gamma$ almost-surely all $z \in Z$ and for all $k \in \N$ we have $\gamma(E_k(z)) > 0$. Then, for any function $\phi : Z \to \extendedR$ which is integrable with respect to $\gamma$, we have that $\E[\phi | \cF_k]$ is constant and precisely given by 
	\begin{equation}
		\E[\phi | \cF_k] = \frac{1}{\gamma(E_k(z))} \int_{E_k(z)} \phi(w) \, \gamma(\dd w)
	\end{equation}
	Now, it is a known fact that $\E[\phi | \cF_k](z) \to \phi(z)$ as $k \to \infty$  for $\gamma$-a.s. all point $z \in Z$. Indeed, this follows from \textit{Levy upwards theorem} --- which is a peculiar case of the celebrated Doob's martingale convergence theorems, see \textit{e.g.} \cite{billingsleyProbabilityMeasure1995}.
\end{proof}
\begin{rem}\label{rem:ldt-gen-tensor-product}
	We remark that the exact same construction yields that for all $\Phi \in L^1(\gamma \otimes \gamma)$, it holds that 
	\begin{equation}
		\frac{1}{\gamma \otimes \gamma (E_k(z) \times E_k(t))} \int_{E_k(z) \times E_k(t)} \Phi d \gamma \otimes \gamma \cto{k \to \infty} \Phi(z, t)
	\end{equation}
	for $\gamma \otimes \gamma$-a.s. all couples $(z, t) \in Z \times Z$. Here, we emphasise that we can consider the \textit{same } $E_k$'s are constructed above. Indeed, one replaces in the above proof $\cF_k$ by the finitely-generated $\sigma$-algebra $$\mathcal{G}_k := \sigma(\{ U_i \times U_j \,|\, i,j=1, \dots, k\})$$ (which in fact is the product $\sigma$-algebra $\cF_k \otimes \cF_k$). The atoms of $\mathcal{G}_k$ are the product of the atoms of $\cF_k$, and the claim follows.
\end{rem}

\bigskip

{\bfseries Acknowledgements.} We acknowledge the financial support of European Research Council (ERC) under the European Union’s
Horizon 2020 Research and Innovation Programme – Grant Agreement n°101077204 HighLEAP. L.N. benefited from the support of the FMJH Program PGMO and from the ANR project GOTA (ANR-23-CE46-0001).

\printbibliography
\end{document}